\pdfoutput=1
\documentclass[11pt]{amsart}
\usepackage{amssymb,latexsym,amsmath,amsthm,amsfonts, enumerate}

\usepackage{color}
\usepackage[all]{xy}
\usepackage{etex}
\usepackage{mathrsfs}

\usepackage{tikz}
\usetikzlibrary{calc,decorations.markings,arrows}
\tikzset{->-/.style={decoration={markings,mark=at position .5 with {\arrow{>}}},postaction={decorate}}}

\usetikzlibrary{positioning,shapes,shadows,arrows,snakes}

\usepackage[colorlinks=true,pagebackref,hyperindex]{hyperref}
\newcommand\myshade{85}
\colorlet{mylinkcolor}{violet}
\colorlet{mycitecolor}{red}
\colorlet{myurlcolor}{cyan}

\hypersetup{
  linkcolor  = mylinkcolor!\myshade!black,
  citecolor  = mycitecolor!\myshade!black,
  urlcolor   = myurlcolor!\myshade!black,
  colorlinks = true,
}

\numberwithin{equation}{section}
\UseRawInputEncoding
\title{Quivers with potentials for Grassmannian cluster algebras}
\author{Wen Chang}\thanks{Wen Chang is supported by the NSF of China (Grant No. 11601295), Shaanxi Province and Shaanxi Normal University.}
\address{School of Mathematics and Statistics, Shaanxi Normal University,
Xi'an 710062, China.}
\email{changwen161@163.com}

\author{Jie Zhang}\thanks{Jie Zhang is supported by the NSF of China (Grant No. 11401022).}
\address{School of Mathematics and Statistics, Beijing institute of technology 100081, Beijing, P. R. China}
\email{jiezhang@bit.edu.cn}
\begin{document}
\maketitle
\begin{abstract}
We consider a quiver with potential $(Q(D),W(D))$ and an iced quiver with potential $(\overline{Q}(D), F(D), \overline{W}(D))$ associated to a Postnikov Diagram $D$ and prove that the mutations of them are compatible with the geometric exchanges of the Postnikov diagram $D$. This ensures we may define quivers with potentials $(Q,W)$ and $(\overline{Q},F,\overline{W})$ for a Grassmannian cluster algebra up to mutation-equivalence.  It shows that $(Q,W)$ is always rigid (thus non-degenerate) and Jacobi-finite. And in fact, it is the unique non-degenerate (thus rigid) quiver with potential by using a general result of Gei\ss-Labardini-Schr\"oer \cite{GLS16}.

Then we show that within the mutation class of the quiver with potential for a Grassmannian cluster algebra,
the quivers determine the potentials up to right equivalence.
As an application, we verify that the auto-equivalence group of the generalized cluster category ${\mathcal{C}}_{(Q, W)}$ is isomorphic to the cluster automorphism group of the associated Grassmannian cluster algebra ${{\mathcal{A}}_Q}$  with trivial coefficients.
\end{abstract}

\def\s{\stackrel}
\def\t{\tilde}

\newcommand{\p}{\scriptstyle}
\newtheorem{Theorem}{Theorem}[section]
\newtheorem{corollary}[Theorem]{Corollary}
\newtheorem{lemma}[Theorem]{Lemma}
\newtheorem{example}{Example}
\newtheorem{proposition}[Theorem]{Proposition}
\newtheorem{definition}[Theorem]{Definition}
\newtheorem{remark}[Theorem]{Remark}
\newtheorem{definition-proposition}[Theorem]{Definition-Proposition}
\newtheorem{conjecture}[Theorem]{Conjecture}
\newcommand{\qihao}{\fontsize{7.25pt}{\baselineskip}\selectfont}
\newcommand{\bahao}{\fontsize{6.25pt}{\baselineskip}\selectfont}
\newcommand{\shihao}{\fontsize{4.25pt}{\baselineskip}\selectfont}
\hyphenation{ap-pro-xi-ma-tion}

\def\Longrightarrow{{\longrightarrow}}
\def\A{\mathcal{A}}
\def\F{\mathcal{F}}
\def\R{\mathcal{R}}
\def\S{\Sigma}
\def\X{\mathscr{X}}
\def\Y{\mathcal{Y}}
\def\x{{\mathbf x}}
\def\p{{\mathbf p}}
\def\P{\mathbb{P}}
\def\ex{{\mathbf{ex}}}
\def\fx{{\mathbf{fx}}}

\def\Gr{\mbox{Gr}}
\def \text{\mbox}

\newcommand{\Z}{\mathbb{Z}}
\newcommand{\Q}{\mathbb{Q}}
\newcommand{\N}{\mathbb{N}}
\newcommand{\C}{\mathbb{C}}
\newcommand{\T}{\mathcal{T}}
\hyphenation{ap-pro-xi-ma-tion}
\newcommand{\NN}{\mathcal{N}}

\newcommand{\CM}{\operatorname{CM}}
\newcommand{\Sub}{\operatorname{Sub}}

\newcommand{\jie}{\color{blue}}
\newcommand{\wen}{\color{purple}}
\tableofcontents

\section*{Introduction}

Since having been introduced by Fomin and Zelevinsky in the year 2000 \cite{FZ02}, cluster algebras have been seeing a tremendous development. It is believed that the coordinate rings of several algebraic varieties related to semisimple groups have cluster structures. This has been verified for various varieties. An important and early example is the Grassmannians \cite{S06}. In this paper, we study the quivers with potentials associated to Grassmannians cluster algebras.

Recall that as a subalgebra of a rational function field, a (skew-symmetric) cluster algebra is generated by {\it cluster variables} in various {\it seeds}, where a seed is a pair consisting of a quiver and a set of indeterminates in the rational function field. Different seeds are related by an operation so called {\it mutation}. In some sense, the rich combinatorial structures on cluster algebras are given by mutations.
There is a representation-theoretic interpretation of quiver mutations given by Derksen, Weyman and Zelevinsky \cite{DWZ08}. They introduced the notion of quivers with potentials and their decorated representations, where potentials can be considered as sum of cycles in the quiver, the mutations of decorated representations can be viewed as a generalization of Bernstein-Gelfand-Ponomarev reflection functors.

On the other hand, the Postnikov diagram $D$, which is a kind of planar graph on a disc, corresponds to a special cluster in a Grassmannian cluster algebra, which consists of Pl\"ucker coordinates. The strands of the diagram cut the disc into some oriented regions and alternating oriented regions. Then the quiver $Q(D)$ of the cluster can be viewed as a kind of dual of the Postnikov diagram, with the alternating oriented regions as the vertices and the crossings of the strands as the arrows. It is proved by Scott in \cite{S06} that the mutation $\mu_a(Q(D))$ of the quiver at some special vertex $a$ is compatible with a transformation $\mu_R(D)$, called geometric exchange, on the Postnikov diagram associated to an alternating oriented quadrilateral cell $R.$ By considering the boundary of the disk as an oriented cycle, we define an iced quiver $(\overline{Q}(D),F)$. Note that the oriented regions yield some minimal cycles in the quiver. Then we define the potential $\overline{W}(D)$ for the quiver as an alternating sum of these minimal cycles. We then get the following theorem which is a generalization of a result in \cite{S06}, see Theorem \ref{compatible} for more details.

\begin{Theorem}\label{thm1}
The geometric exchanges $\mu_R(D)$ of the Postnikov diagram $D$ are compatible with the mutations of $(Q,W)$ and $(\overline{Q}(D),F,\overline{W}(D))$ up to right equivalence.
\end{Theorem}
Note that the concept of the mutation of an iced quiver with potential we used here is recently induced by Pressland in \cite{P18}. We mention that excepting the related works of Scott, there already exist some other related results which compare the mutations with other operations, like that stated in the above theorem. For example, in \cite{V09}, Vit\'{o}ria compared the mutation of the quiver with potential and the Seiberg duality, where the concept of Seiberg duality is defined by requiring certain complexes to be tilting, and in the context of Postnikov diagrams, this duality actually corresponds to the geometric exchange, cf. \cite{HHJPRR12}. On the other hand, Buan-Iyama-Reiten-Smith proved in \cite{BIRS11} the compatibility of the mutations of cluster tilting objects in the generalized cluster category arising from $(Q,W)$ and the mutations of the quiver with potentials, while Pressland considered in \cite{P18} the case for the iced quiver with potential $(\overline{Q}(D),F,\overline{W}(D))$. And Baur-King-Marsh proved that the dimer algebra arising from a Postnikov diagram is invariant under the geometric exchange \cite{BKM16}.

The above Theorem \ref{thm1} allows us to define the quivers with potentials (up to right equivalence and mutation equivalence) for a Grassmannian cluster algebra which is independent on the choice of the Postnikov diagram.
Note that the mutation of a quiver with potential can only be operated at a vertex which is not involved in 2-cycles, and even when the initial quiver in a quiver with potential has no 2-cycles, there may appear 2-cycles after mutations. A quiver with potential is called {\it non-degenerate} if there exist no 2-cycles after any iterated mutations. A more ``generic" condition called {\it rigidity} implies the non-degeneration. So a rigid quiver with potential can be viewed as a kind of {\it ``good"} quiver with potential, respecting to the mutations. We also study the rigidity of the quiver with potential of a Grassmannian cluster algebra: see Theorem \ref{thm:uniqueness}

\begin{Theorem}
The quiver with potential without frozen vertices $(Q,W)$ associated to a Grassmannian cluster algebra is rigid, and  it is the unique rigid quiver with potential with underlying quiver $Q$ up to right equivalence and mutation equivalence.
\end{Theorem}

We should mention that the rigidity of $(Q,W)$ is already obtained by several authors, for example by Buan-Iyama-Reiten-Smith \cite{BIRS11} and Kulkarni \cite{K19}, where the method used in \cite{BIRS11} is more algebraic, while the method used in \cite{K19} is more topological, like this paper. Compare with  \cite{K19}, the method we use in this paper is different. In fact, Kulkarni got the rigidity for a larger class of quivers with potentials arising from dimer models, while we get the conclusion by giving an explicit description of a special initial quiver with potential. These description are important to us, we will also use them to prove the uniqueness of the rigid quiver with potential of a Grassmannian cluster algebra.

We also mention that the techniques used in subsection \ref{sec:rigidity} to describe the properties of the quivers with potentials associated to Grassmannian cluster algebras is inspired by the work of Labardini for the surface cluster algebras \cite{L09,L16}. The philosophy behind this is that as for the surface cluster algebras, some quivers of the Grassmannian cluster algebras is ``2-dimensional", that means it can be embedded into a disk, noticing that such quivers are the dual of the Postinkov diagrams. So from this point of view, our main results in section \ref{Section:quivers with potentials of Grassmannian cluster algebras}, especially the uniqueness of the rigid quiver with potential, and thus the following application to the cluster automorphism groups, can be established in more general settings, for example, for the cluster algebras arising from the dimer models \cite{B12}, from the unipotent groups \cite{BIRS09}, and from the double Bruhat cells \cite{FZ07}. However, we restrict our interests to the Grassmannian cluster algebras in this paper.

We can easily get the following corollary from the rigidity uniqueness:
\begin{corollary}
Inside the mutation-equivalent class of quiver with potential of a Grassmannian cluster algebra, the quiver determines the potentials up to right equivalent. More precisely, assume that $(Q',W')$ and $(Q,W)$ are two quivers with potentials of a Grassmannian cluster algebra. Then

(1) $(Q',W')$ is right equivalent to $(Q,W)$ if $Q'\cong Q$;

(2) $(Q',W')$ is right equivalent to $(Q^{op},W^{op})$ if $Q'\cong Q^{op}$.
\end{corollary}


As an application of above result, we compare two groups associated to the Grassmannian cluster algebras.
One is the cluster automorphism group, which is introduced in \cite{ASS12} to describe the symmetries of a cluster algebra. It is proved in \cite{ASS12,BIRS09} that if the cluster algebra is of acyclic type, then the cluster automorphism group is isomorphic to the auto-equivalence group of the corresponding cluster category.
We provide a similar isomorphism between these two groups for the Grassmannian cluster algebra with trivial coefficients, see Theorem \ref{thm:auto-equi-gp=clu-auto-gp}.
\begin{Theorem}
Let $(Q,W)$ be a quiver with potential for a Grassmannian cluster algebra. Then the auto-equivalence group of the generalized cluster category $\mathcal{C}_{(Q,W)}$ is isomorphic to the cluster automorphism group
of the associated Grassmannian cluster algebra $\mathcal{A}_{Q}$ with trivial coefficients.
\end{Theorem}

Note that excepting some special Grassmannian cluster algebras, most of them are non-acyclic. In fact, we prove a more general result, this isomorphism is valid for a generalized cluster category whose potentials are determined by the quivers.
Note that on that one hand, these two groups both describe the symmetries of the cluster structures in the category and the algebra respectively. On the other hand, the cluster structure in the cluster algebra $\mathcal{A}_{Q}$ only depends on the quiver, rather than the potential over the quiver.
So we conjecture that the quivers always determine the potentials in the mutation-equivalent classes of quivers with potentials, and thus these two groups are isomorphic for all generalized cluster categories, see Conjecture \ref{conj:quiver-determines-potential} and Conjecture \ref{conj:auto-equi-gp=clu-auto-gp} respectively.

The paper is organized as follows: In section \ref{sec:prelim.}, we recall some preliminaries on cluster algebras, quivers with potentials and Grassmannian cluster algebras. In section \ref{Section:quivers with potentials of Grassmannian cluster algebras}, we define the quivers with potentials for Grassmannian cluster algebras and prove their rigidity and uniqueness. Section \ref{sec:application} is devoted to an application of our main results to the generalized cluster categories, namely, we prove the isomorphism between the auto-equivalence group of the category and the cluster automorphism group in subsection  \ref{Section:auto-equivalence groups and cluster automorphism groups}.

\section*{Conventions}
Throughout the paper, we use $\Z$ as the set of integers, $\N$ as the set of positive integers, and $\C$ as the set of complex numbers. Arrows in a quiver are composed from right to left, that is, we write a path $j\s{\beta}\rightarrow i \s{\alpha}\rightarrow k$ as $\alpha\beta$.

\section{Preliminaries}\label{sec:prelim.}
In this section, we briefly recall some definitions on quivers with potentials and Grassmannian cluster algebras.
\subsection{Quivers with potentials}\label{subsec:QP}
The references of this subsection are \cite{BIRS11,DWZ08,GLS16,P18}, especially \cite{P18} for the case of iced quiver with potentials.

\bigskip

{\it{Quivers:}}
\noindent

Recall that a {\it quiver} is a quadruple $Q=(Q_0,Q_1,s,t)$ consisting of a finite set of {\it vertices} $Q_0$, of a finite set of {\it arrows} $Q_1$, and of two maps $s, t$ from $Q_1$ to $Q_0$ which map each arrow $\alpha$ to its {\it source} $s(\alpha)$ and its {\it target} $t(\alpha)$, respectively. An {\it iced quiver} is a pair $(Q,F)$ where $Q$ is a quiver and $F=(F_0,F_1,s,t)$ is a subquiver (not necessarily full) of $Q$, where $F_0\subseteq Q_0$ and $F_1\subseteq Q_1$.
The vertices in $F_0$ are called the {\it frozen vertices}, while the vertices in $Q_0\setminus F_0$ are called the {\it exchangeable vertices}.
The arrows in $F_1$ are called the {\it frozen arrows}, while the arrows in $Q_1\setminus F_1$ are called the {\it unfrozen arrows}.
The full subquiver of $Q$ with vertex set $Q_0\setminus F_0$ is called the {\it principal part} of $Q$, denoted by $Q^{pr}$.






Let $(Q,F)$ be an iced quiver without loops nor 2-cycles. A mutation of $(Q,F)$ at exchangeable vertex $i$ is an iced quiver $(\mu_i(Q),F)$, where $\mu_i(Q)$ is obtained from $Q$ by:
\begin{itemize}
\item inserting a new unfrozen arrow $\gamma: j\to k$ for each path $j\s{\beta}\rightarrow i \s{\alpha}\rightarrow k$;
\item inverting all arrows passing through $i$;
\item removing the arrows in a maximal set of pairwise disjoint $2$-cycles ({\it $2$-cycles moves}).
\end{itemize}

\bigskip
{\it{Cluster algebras:}}

Let $(Q,F)$ be an iced quiver with $Q_0 = \{1, 2, \ldots, n+m\}$ and $F_0 = \{n+1, n+2, \ldots, n+m\}$.
By associating to each vertex $i \in Q_0$ an indeterminate element $x_i$, one gets a set $\tilde{\x}=\{x_1, x_2,\ldots, x_{n+m}\}=\{x_1, x_2, \ldots, x_{n}\}\sqcup\{x_{n+1}, x_{n+2}, \ldots, x_{n+m}\}=\x\sqcup\p$. We call the triple $\S=(Q,F,\tilde{\x})$ a {\it{seed}}. An element in $\x$ (resp. in $\p$) is called a {\it{cluster variable}} (resp. {\it{coefficient variable}}), and $\x$ is called a {\it{cluster}}.

Let $x_i$ be a cluster variable, the mutation of the seed $\S$ at $x_i$ is a new seed $\mu_i(\S)=(\mu_i(Q),F,\mu_i(\x))$, where $\mu_i(\x) = (\x \setminus \{x_i\}) \sqcup \{x'_i\}$ with
\begin{equation}
\label{eq: exchange relations}
x_i x_i^\prime = \prod_{\substack{\alpha\in Q_1; \\ s(\alpha)=i}} x_{t(\alpha)} + \prod_{\substack{\alpha\in Q_1; \\ t(\alpha)=i}} x_{s(\alpha)}.
\end{equation}

Denote by ${\X}$ the union of all possible clusters obtained from an initial seed $\S=(Q,F,\tilde{\x})$ by iterated mutations.
Let $\P$ be the free abelian group (written multiplicatively) generated by the elements of $\p$.
Let $\F=\Q \P(x_{1},x_{2},\ldots ,x_{n})$ be the field of rational functions in $n$ independent
variables with coefficients in $\Q \P$.
The {\it cluster algebra} $\A_{(Q,F)}$ is a $\Z\P$-subalgebra of $\F$ generated by cluster variables in ${\X}$, that is
\[\A_{(Q,F)}=\Z\P[\X].\]
\bigskip

{\it{Quivers with potentials:}}

Let $(Q,F)$ be an iced quiver without loops. We denote by $\C\langle Q\rangle$ the {\it{path algebra}} of $Q$ over $\C$. By $length(p)$ we denote the {\it{length}} of a path $p$ in $\C\langle Q\rangle$. The {\it{complete path algebra}} $\C\langle \langle Q\rangle\rangle$ is the completion of $\C\langle Q\rangle$ with respect to the ideal $\mathfrak{m}$
generated by the arrows of $Q$. A {\it{potential}} $W$ on $Q$ is an element in the closure $Pot(Q)$ of the space generated by all cycles in $Q$.
We say two potentials $W$ and $W'$ are \emph{cyclically
equivalent} if $W-W'$ belongs to the closure $C$ of the space
generated by all differences $\alpha_s \cdots\alpha_2 \alpha_1 - \alpha_1\alpha_{s} \cdots \alpha_2$ coming from
cycles $\alpha_s \cdots \alpha_2\alpha_1$. Denote by $[l]$ the set of cycles which are cyclically equivalent to a cycle $l$.
We call a triple $(Q,F,W)$ an {\it{iced quiver with potential}} or an IQP for brevity, if no two terms in $W\in Pot(Q)$ are cyclically equivalent. Moreover, if each term in $W$ includes at least one unfrozen arrow, then we call the IQP {\it irredundant}. If $F=\emptyset$, then we call the pair $(Q,W)$ a {\it{quiver with potential}} or a QP for brevity, and as for the quiver, we also view a QP as a special IQP.

Let $(Q,F,W)$ and $(Q',F',W')$ be two IQPs with $Q_0=Q'_0$. Their {\em direct sum}, denoted by $(Q\oplus Q',F\cup F',W+W')$, is a new IQP, where $Q\oplus Q'$ is a quiver with vertex set $Q_0$ and arrow set $Q_1\sqcup Q'_1$, $F\cup F'$ is a subquiver of $Q\oplus Q'$ with vertex set $F_0\cup F'_0$ and arrow set $F_1\cup F'_1$.


\bigskip

{\it{Jacobian algebras:}}

For an arrow $\alpha$ of $Q$, we define $\partial_\alpha : Pot(Q) \to \C\langle \langle Q\rangle\rangle$ the {\em cyclic derivative} with respect to $\alpha$, which is the unique continuous linear map that sends a cycle $l$ to the sum $\sum_{l=p\alpha q} pq$ taken over all decompositions of the cycle $l$.
Let $J(Q,F,W)$ be the closure of the ideal of $\C\langle Q\rangle$ generated by cyclic derivatives in
$\{\partial_\alpha W,\alpha ~{\textrm {unfrozen}} \}.$
We call $J(Q,F,W)$ the {\it{(frozen) Jacobian ideal}} of $(Q,F,W)$ and call the quotient
$$\mathcal{P}(Q,F,W)=\C\langle \langle Q\rangle\rangle/J(Q,F,W)$$
the {\it{(frozen) Jacobian algebra}} of $(Q,F,W)$.

For an IQP $(Q,F,W)$, we call it \emph{trivial} if each term in $W$ is a $2$-cycle and $\mathcal{P}(Q,F,W)$ is a product of copies of $\C$, we say it is \emph{reduced} if each term of $W$ includes at least one unfrozen arrow and $\partial_\beta W\in \mathfrak{m}^2$ for any unfrozen arrow $\beta$.

\bigskip
{\it{Right-equivalences and reductions:}}

Two IQPs $(Q,F,W)$ and $(Q',F',W')$ are \emph{right-equivalent} if $Q$ and
$Q'$ have the same set of vertices and frozen vertices, and there exists an algebra
isomorphism $\varphi:\C\langle \langle Q\rangle\rangle\rightarrow \C\langle \langle Q'\rangle\rangle$ whose
restriction on vertices is the identity map, $\phi(\C\langle \langle F\rangle\rangle)=\C\langle \langle F'\rangle\rangle$, and $\varphi(W)$ and
$W$ are cyclically equivalent. Such an isomorphism $\varphi$ is
called a \emph{right-equivalence}.

It is proved in \cite[Theorem 3.6]{P18} (in \cite[Theorem 4.6]{DWZ08} that for QP) that for any irredundant IQP $(Q,F,W)$, there exist a reduced
IQP $(Q_{\mathrm{red}},F_{\mathrm{red}},W_{\mathrm{red}})$ such that the Jacobian algebras $\mathcal{P}(Q,F,W)$ and $\mathcal{P}(Q_{\mathrm{red}},F_{\mathrm{red}},W_{\mathrm{red}})$ are isomorphic. Furthermore, the
right-equivalence class of
$(Q_{\mathrm{red}},F_{\mathrm{red}},W_{\mathrm{red}})$ is determined by the
right-equivalence class of $(Q,F,W)$.
The operation to producing $(Q_{\mathrm{red}},F_{\mathrm{red}},W_{\mathrm{red}})$ is called the {\it reduction}, which consisting of the following steps (see Lemma 3.14 and the proof of Theorem 3.6 \cite{P18}):

step I: up to replacing $W$ by a right equivalent potential, we write
\begin{equation}
\label{eq:potential}
W=\sum_{i=1}^M\alpha_i\beta_i+\sum_{i=M+1}^N\alpha_i(\beta_i+p_i)
+W_1
\end{equation}
for some arrows $\alpha_i$ and $\beta_i$ and elements $p_i\in \mathfrak{m}^2$, where
\begin{itemize}
\item $\alpha_i$ is unfrozen for all $1\leq i\leq N$, and $\beta_i$ is frozen if and only if $i>M$,
\item the arrows $\alpha_i$ and $\beta_i$ with $1\leq i\leq M$ each appear exactly once in the expression,
\item the arrows $\beta_i$, for $1\leq i\leq N$, do not appear in any of the $p_j$, and
\item the arrows $\alpha_i$ and $\beta_i$, for $1\leq i\leq N$, do not appear in the potential $W_1$, which has no length $2$ terms.
\end{itemize}

Step II: Let $Q'$ be the subquiver of $Q$ consisting of all vertices and those arrows which are not $\alpha_i$ and $\beta_i, 1\leq i\leq N$, and
$$W'=\sum_{i=M+1}^N\alpha_i(\beta_i+p_i)
+W_1.$$
Then $(Q',F,W')$ is an IQP.

Step III: Let $(Q_{\mathrm{red}},F_{\mathrm{red}})$ be the ice quiver obtained from $(Q',F)$ by deleting $\beta_i$ and freezing $\alpha_i$ for each $M+1\leq i\leq N$. Let
$$W_{\mathrm{red}}=\sum_{i=M+1}^N \alpha_ip_i
+W_1.$$
Then $(Q_{\mathrm{red}},F_{\mathrm{red}},W_{\mathrm{red}})$ is the reduced IQP we want.

\bigskip

{\it{Mutations of iced quivers with potentials:}}

Let $(Q,F,W)$ be an irredundant IQP, and let $i$ be an exchangeable vertex of $Q$ such that there is no $2$-cycles at $i$ and no cycle occurring in the decomposition of $W$ starts and ends at $i$.
The {\em pre-mutation} $\widetilde{\mu}_i(Q,F,W)$ of $(Q,F,W)$ is a new quiver with potential $(\widetilde{\mu}_i(Q),\widetilde{\mu}_i(F),\widetilde{\mu}_i(W))=(\widetilde{Q},\widetilde{F},\widetilde{W})$  defined as follows. The new iced quiver $(\widetilde{Q},\widetilde{F})$ is obtained from $(Q,F)$ by
\begin{itemize}
\item adding a new unfrozen arrow $[\alpha\beta]: j\to k$ for each path $j\s{\beta}\rightarrow i \s{\alpha}\rightarrow k$;
\item replacing each arrow $\alpha$ incident to $i$ with an arrow $\alpha^*$ in the opposite direction.
\end{itemize}
The new potential $\widetilde{W}$ is the sum of two potentials $\widetilde{W}_1$ and $\widetilde{W}_2$. The potential $\widetilde{W}_1$ is obtained by replacing each factor $\alpha_p\alpha_{p+1}$ by  $[\alpha_p\alpha_{p+1}]$ with $s(\alpha_p)=t(\alpha_{p+1})=i$ for any cyclic path $\alpha_1\cdots\alpha_s$ occuring in the expansion of $W.$ The potential $\widetilde{W}_2$ is given by
$$\widetilde{W}_2=\sum_{\alpha,\beta}[\alpha\beta]\beta^{*}\alpha^{*},$$
where the sum ranges over all pairs of arrows $\alpha$ and $\beta$
 with $s(\alpha)=t(\beta)=i$. Then $\widetilde{\mu}_i(Q,F,W)$ is an irredundant IQP.
We denote by $\mu_i(Q,F,W)$ the reduced part of $\widetilde{\mu}_i(Q,F,W)$, and call $\mu_i$ the \emph{mutation} of $(Q,F,W)$ at the vertex $i$. We call two IQPs {\em mutation equivalent} if one can be obtained from another by iterated mutations. Note that the mutation equivalence is an equivalent relation on the set of right-equivalence classes of IQPs.



\subsection{Grassmannian cluster algebras} We recall in this subsection some definitions on Grassmannian cluster algebras, and we refer to \cite{P06} and \cite{S06} for more details on Postnikov diagrams and Grassmannian cluster algebras, respectively.

Let $\Gr(k,n)$ be the Grassmannian of $k$-planes in $\C^n$ and $\C[\Gr(k,n)]$ be its homogeneous coordinate ring. When $k=2$, Fomin and Zelevinsky proved
that $\C[\Gr(k,n)]$ has a cluster algebra structure \cite{FZ03}.
Scott generalised this result to the case of any Grassmannian $\Gr(k,n)$, where the proof relies on a correspondence between a special kind of clusters in the cluster algebra and a kind of planar diagram.

\bigskip
{\it Postnikov diagrams:}

For $k, n\in \N$ with $k < n$, a {\it{$(k,n)$-Postnikov diagram}} $D$ is a collection of $n$ oriented paths, called {\it{strands}}, in a disk with $n$ marked points on its boundary, labeled by $1, 2, \ldots, n$ in clockwise orientation. The strands, which are labeled by $1 \leqslant i \leqslant n$, start at point $i$ and end at point $i+k$. These strands obey the following conditions:

\begin{itemize}
\item Any two strands cross transversely, and there are no triple crossings between strands.
\item No strand intersects itself.
\item There are finitely many crossing points.
\item Following any given strand, the other strands alternately cross it  from left to right and from right to left.
\item For any two strands $i$ and $j$ the configuration shown in Figure \ref{fig:forbidden} is forbidden.
\end{itemize}

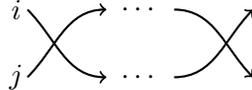
\begin{figure}[ht]
\begin{center}
{\begin{tikzpicture}[scale=0.15]

\node[] (C) at (0,3)
						{$\cdots$};
\node[] (C) at (0,-3)
						{$\cdots$};

\draw[->,thick] (3,3) to [out=0,in=135](10,-3);
\draw[->,thick] (3,-3) to [out=0,in=-135](10,3);
\draw[->,thick] (-10,-3) to [out=45,in=180](-3,3);
\draw[->,thick](-10,3) to [out=-45,in=180](-3,-3);

\node[] (C) at (-11,3)
						{$i$};
\node[] (C) at (-11,-3)
						{$j$};
\end{tikzpicture}}
\caption{Forbidden crossing}
\label{fig:forbidden}
\end{center}
\end{figure}

Postnikov diagrams are identified up to isotopy.
We say that a Postnikov diagram is of reduced type if no {\it{untwisting move}} shown in Figure \ref{Figure:untwisting move} can be applied to it.
\begin{figure}[ht]
\begin{center}
{\begin{tikzpicture}[scale=0.15]
\draw[thick] (-3,3) -- (3,3);
\draw[thick] (-3,-3) -- (3,-3);

\draw[thick] (3,3) to [out=0,in=135](10,-3);
\draw[thick] (3,-3) to [out=0,in=-135](10,3);
\draw[thick] (-10,-3) to [out=45,in=180](-3,3);
\draw[thick](-10,3) to [out=-45,in=180](-3,-3);

\node[] (C) at (-11,3)
						{$i$};
\node[] (C) at (-11,-3)
						{$j$};

\draw[->,thick] (-0.1,-3) to (0,-3);
\draw[<-,thick] (0,3) to (0.1,3);

\draw[->,thick] (18,0) to (22,0);
\end{tikzpicture}}
\qquad\qquad
{\begin{tikzpicture}[scale=0.15]
\draw[thick] (-10,3) -- (10,3);
\draw[thick] (-10,-3) -- (10,-3);

\node[] (C) at (-11,3)
						{$i$};
\node[] (C) at (-11,-3)
						{$j$};

\draw[->,thick] (-0.1,3) to (0,3);
\draw[<-,thick] (0,-3) to (0.1,-3);
\end{tikzpicture}}
\caption{Untwisting move}
\label{Figure:untwisting move}
\end{center}
\end{figure}
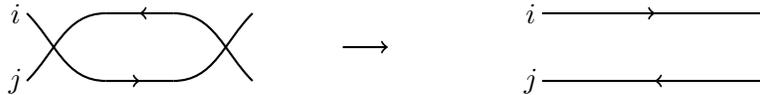
The fourth condition ensures that the strands divide the disc into two types of regions: {\it{oriented regions}}, where all the strands on their boundary circle clockwise or counterclockwise, and {\it{alternating oriented regions}}, where the adjacent strands alternate directions.
A region is said to be {\it{internal}} if it is not adjacent to the boundary of the disk, and the other regions are referred to as {\it{boundary regions}}. See Figure \ref{Example-Figure:iced quiver} for an example of $(3,7)$-Postnikov diagram.

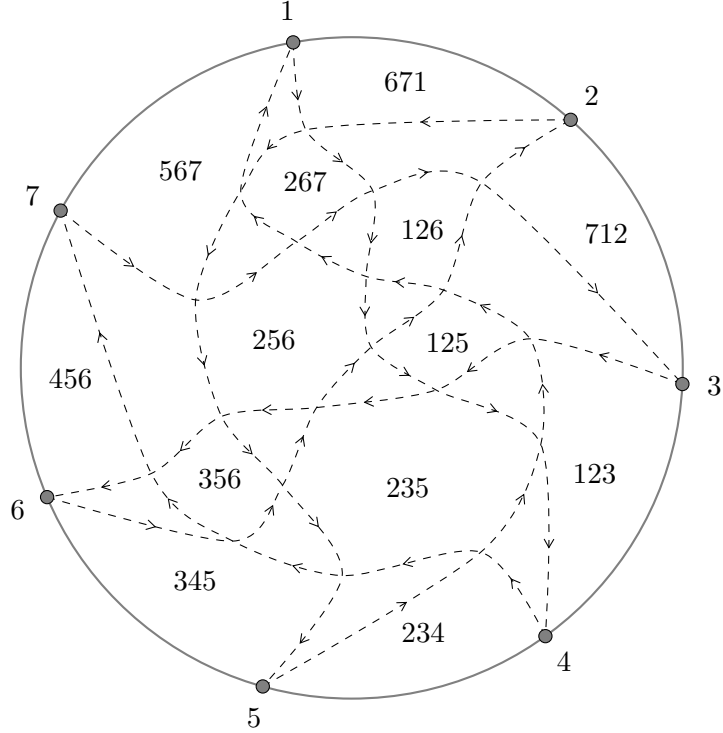
\begin{figure}
\[
\begin{tikzpicture}[scale=1.1,baseline=(bb.base),
 strand/.style={black,dashed}]

\newcommand{\strarrow}{\arrow{angle 60}}
\newcommand{\bstart}{100} 
\newcommand{\seventh}{51.4} 
\newcommand{\qstart}{150.7} 
\newcommand{\bigrad}{4cm} 
\newcommand{\eps}{11pt} 
\newcommand{\dotrad}{0.1cm} 
\newcommand{\bdrydotrad}{{0.8*\dotrad}} 

\path (0,0) node (bb) {};


\draw (0,0) circle(\bigrad) [thick,gray];

{ \coordinate (b1) at (\bstart-\seventh*1:\bigrad);
  \draw (\bstart-\seventh*1:\bigrad+\eps) node {$2$}; }

{ \coordinate (b2) at (\bstart-\seventh*2:\bigrad);
  \draw (\bstart-\seventh*2:\bigrad+\eps) node {$3$}; }

{ \coordinate (b3) at (\bstart-\seventh*3:\bigrad);
  \draw (\bstart-\seventh*3:\bigrad+\eps) node {$4$}; }

{ \coordinate (b4) at (\bstart-\seventh*4:\bigrad);
  \draw (\bstart-\seventh*4:\bigrad+\eps) node {$5$}; }

{ \coordinate (b5) at (\bstart-\seventh*5:\bigrad);
  \draw (\bstart-\seventh*5:\bigrad+\eps) node {$6$}; }

{ \coordinate (b6) at (\bstart-\seventh*6:\bigrad);
  \draw (\bstart-\seventh*6:\bigrad+\eps) node {$7$}; }

{ \coordinate (b7) at (\bstart-\seventh*7:\bigrad);
  \draw (\bstart-\seventh*7:\bigrad+\eps) node {$1$}; }

\foreach \n/\a/\r in {8/77/0.79, 10/130/0.5, 12/60/0.2, 14/350/0.5, 16/220/0.3, 18/225/0.75, 20/280/0.75,
    9/117/0.77, 11/92/0.38, 13/30/0.7, 15/290/0.05, 17/185/0.7, 19/250/0.55,  21/330/0.75}
{ \coordinate (b\n)  at (\a:\r*\bigrad); }


\foreach \e/\f/\t in {8/9/0.4, 9/10/0.5, 8/11/0.4, 10/11/0.5, 11/12/0.5, 8/13/0.5, 12/13/0.65, 13/14/0.4, 12/15/0.5,
 14/15/0.5, 15/16/0.5, 16/17/0.65, 10/17/0.6, 17/18/0.45, 18/19/0.5, 19/20/0.5, 20/21/0.5, 16/19/0.5, 14/21/0.5}
{\coordinate (a\e\f) at (${\t}*(b\e) + {1-\t}*(b\f)$); }


\draw [strand] plot[smooth]
coordinates {(b1) (a89) (a910) (a1017) (a1617) (a1619) (a1920) (b4)}
[ postaction=decorate, decoration={markings,
  mark= at position 0.16 with \strarrow, mark= at position 0.33 with \strarrow,
  mark= at position 0.46 with \strarrow, mark= at position 0.58 with \strarrow,
  mark= at position 0.69 with \strarrow, mark= at position 0.795 with \strarrow,
  mark= at position 0.94 with \strarrow }];
\draw [strand] plot[smooth] coordinates {(b2) (a1314) (a1415) (a1516) (a1617) (a1718) (b5)}
[ postaction=decorate, decoration={markings,
  mark= at position 0.13 with \strarrow, mark= at position 0.34 with \strarrow,
  mark= at position 0.50 with \strarrow, mark= at position 0.65 with \strarrow,
  mark= at position 0.79 with \strarrow, mark= at position 0.92 with \strarrow }];
\draw [strand] plot[smooth] coordinates {(b3) (a2021) (a1920) (a1819) (a1718) (b6)}
[ postaction=decorate, decoration={markings,
  mark= at position 0.09 with \strarrow, mark= at position 0.25 with \strarrow,
  mark= at position 0.40 with \strarrow, mark= at position 0.59 with \strarrow,
  mark= at position 0.83 with \strarrow }];
\draw [strand] plot[smooth] coordinates {(b4) (a2021) (a1421) (a1314) (a1213) (a1112) (a1011) (a910) (b7)}
[ postaction=decorate, decoration={markings,
  mark= at position 0.17 with \strarrow, mark= at position 0.335 with \strarrow,
  mark= at position 0.45 with \strarrow, mark= at position 0.56 with \strarrow,
  mark= at position 0.65 with \strarrow, mark= at position 0.73 with \strarrow,
  mark= at position 0.81 with \strarrow, mark= at position 0.93 with \strarrow, }];
\draw [strand] plot[smooth] coordinates {(b5) (a1819) (a1619) (a1516) (a1215) (a1213) (a813) (b1)}
[ postaction=decorate, decoration={markings,
  mark= at position 0.15 with \strarrow, mark= at position 0.33 with \strarrow,
  mark= at position 0.43 with \strarrow, mark= at position 0.55 with \strarrow,
  mark= at position 0.65 with \strarrow, mark= at position 0.78 with \strarrow,
  mark= at position 0.93 with \strarrow }];
  \draw [strand] plot[smooth] coordinates {(b6) (a1017) (a1011) (a811) (a813) (b2)}
[ postaction=decorate, decoration={markings,
  mark= at position 0.12 with \strarrow, mark= at position 0.30 with \strarrow,
  mark= at position 0.43 with \strarrow, mark= at position 0.56 with \strarrow,
  mark= at position 0.84 with \strarrow }];
\draw [strand] plot[smooth] coordinates {(b7) (a89) (a811) (a1112) (a1215) (a1415) (a1421) (b3)}
[ postaction=decorate, decoration={markings,
  mark= at position 0.08 with \strarrow, mark= at position 0.19 with \strarrow,
  mark= at position 0.32 with \strarrow, mark= at position 0.42 with \strarrow,
  mark= at position 0.53 with \strarrow, mark= at position 0.66 with \strarrow,
  mark= at position 0.88 with \strarrow }];


\foreach \n in {1,2,3,4,5,6,7} {\draw (b\n) circle(\bdrydotrad) [fill=gray];}


\foreach \n/\m/\r in {1/671/0.88, 2/712/0.87, 3/123/0.8, 4/234/0.83, 5/345/0.8, 6/456/0.85, 7/567/0.79}
{ \draw (\qstart-20-\seventh*\n:\r*\bigrad) node (q\m) {$\m$}; }

\foreach \m/\a/\r in {267/104/0.58 , 126/63/0.47, 256/160/0.25, 125/15/0.3, 356/220/0.52, 235/295/0.4 }
{ \draw (\a:\r*\bigrad) node (q\m) {$\m$}; }

 \end{tikzpicture}
\]
\caption{A $(3,7)$-Postnikov diagram}
\label{Example-Figure:iced quiver}
\end{figure}

Given a Postnikov diagram $D$ and an {\it{alternating oriented quadrilateral cell}} $R$ inside $D$, a new Postnikov diagram $\widetilde{\mu}_{R}(D)$ is constructed by the local rearrangement shown in Figure \ref{Figure:geometric exchange}. We call $\widetilde{\mu}_{R}$ a {\it{pre-geometric exchange}} at $R$. Note that there may appear new configurations in $\widetilde{\mu}_{R}(D)$ as shown in the left side of Figure \ref{Figure:untwisting move}. Let $\mu_{R}(D)$ be the Postnikov diagram obtained from $\widetilde{\mu}_{R}(D)$ after untwisting these new configurations. We call $\mu_{R}$ a {\it{geometric exchange}} at $R$.
Note that if $D$ is of reduced type then so is $\mu_{R}(D)$.

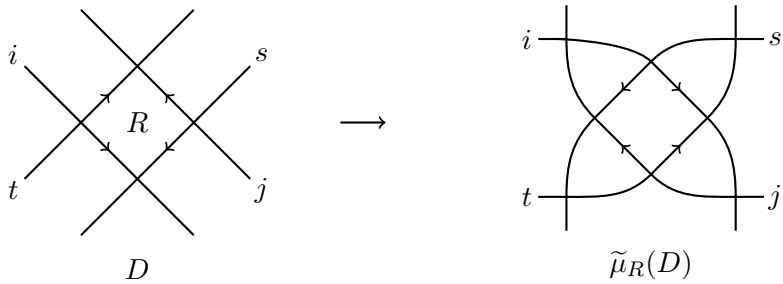
\begin{figure}
\begin{center}
{\begin{tikzpicture}[scale=0.15]
\draw[thick] (-10,-5) -- (5,10);
\draw[thick](-5,-10) to (10,5);
\draw[thick] (10,-5) -- (-5,10);
\draw[thick](5,-10) to (-10,5);

\draw[->,thick] (18,0) to (22,0);

\node[] (C) at (11,-6)
						{$j$};
\node[] (C) at (-11,6)
						{$i$};
\node[] (C) at (11, 6)
						{$s$};
\node[] (C) at (-11,-6)
						{$t$};

\draw[->,thick] (-2.6,2.4) to (-2.5,2.5);
\draw[->,thick] (-2.6,-2.4) to (-2.5,-2.5);
\draw[->,thick] (2.6,2.4) to (2.5,2.5);
\draw[->,thick] (2.6,-2.4) to (2.5,-2.5);
\node[] (C) at (0,0)
						{$R$};
\node[] (C) at (0,-13)
						{$D$};
\end{tikzpicture}}
\qquad\qquad
{\begin{tikzpicture}[scale=0.15]
\draw[thick] (-5,0) -- (0,5);
\draw[thick](0,-5) to (5,0);
\draw[thick] (5,0) -- (0,5);
\draw[thick](0,-5) to (-5,0);

\draw[thick] (-5,0) to [out=-135,in=90](-7.5,-7);
\draw[thick](0,5) to [out=45,in=180](7.5,7);
\draw[thick] (-7.5,-7) to [out=0,in=-135] (0,-5);
\draw[thick] (7.5,7) to [out=-90,in=45] (5,0);

\draw[thick] (-7.5,-7)to(-10,-7) ;
\draw[thick] (-7.5,-7)to(-7.5,-10) ;
\draw[thick] (7.5,7)to(10,7) ;
\draw[thick] (7.5,7)to(7.5,10) ;

\draw[thick](7.5,-7) to [out=90,in=-45](5,0);
\draw[thick](-7.5,7)to [out=180,in=135](0,5) ;
\draw[thick](0,-5) to [out=-45,in=180] (7.5,-7);
\draw[thick] (-5,0) to [out=135,in=-90](-7.5,7);

\draw[thick] (7.5,-7)to(10,-7) ;
\draw[thick] (7.5,-7)to(7.5,-10) ;
\draw[thick] (-7.5,7)to(-10,7) ;
\draw[thick] (-7.5,7)to(-7.5,10) ;

\node[] (C) at (11,-7)
						{$j$};
\node[] (C) at (-11,7)
						{$i$};
\node[] (C) at (11, 7)
						{$s$};
\node[] (C) at (-11,-7)
						{$t$};

\draw[<-,thick] (-2.6,2.4) to (-2.5,2.5);
\draw[<-,thick] (-2.6,-2.4) to (-2.5,-2.5);
\draw[<-,thick] (2.6,2.4) to (2.5,2.5);
\draw[<-,thick] (2.6,-2.4) to (2.5,-2.5);
\node[] (C) at (0,-13)
						{$\widetilde{\mu}_R(D)$};
\end{tikzpicture}}
\caption{Pre-geometric exchange}
\label{Figure:geometric exchange}
\end{center}
\end{figure}

\bigskip
{\it{Grassmannian cluster algebras:}}
For a Postnikov diagram $D$, one may associate it a quiver $\overline{Q}(D)$, whose vertices are labeled by the alternating oriented regions of $D$ and arrows correspond to intersection points of two alternating regions, with orientation as in Figure \ref{Figure: Orientation convention for the quiver}. Let $(\overline{Q}(D),F(D))$ be an iced quiver, where the internal alternating oriented regions correspond to the exchangeable vertices of $\overline{Q}(D)$, while the boundary alternating oriented regions correspond to the frozen vertices, and the frozen arrows are all the arrows connecting the boundary alternating oriented regions, which are red arrows in Figure \ref{Figure: Orientation convention for the quiver}.
We denote by $Q(D)$ the principal part of $\overline{Q}(D)$.

\begin{example}

See Figure \ref{Example-Figure:iced quiver2} for the quiver $(\overline{Q}(D),F(D))$ associated to the Postnikov diagram $D$ in Figure~\ref{Example-Figure:iced quiver}.

\begin{figure}
\[
\begin{tikzpicture}[scale=1.1,baseline=(bb.base),
 strand/.style={black, dashed},
 quivarrow/.style={black, -latex, thick},
 equivarrow/.style={red, -latex, thick},
 eequivarrow/.style={red, -latex, thick}]
\newcommand{\strarrow}{\arrow{angle 60}}
\newcommand{\bstart}{100} 
\newcommand{\seventh}{51.4} 
\newcommand{\qstart}{150.7} 
\newcommand{\bigrad}{4cm} 
\newcommand{\eps}{11pt} 
\newcommand{\dotrad}{0.1cm} 
\newcommand{\bdrydotrad}{{0.8*\dotrad}} 

\path (0,0) node (bb) {};

{ \coordinate (b1) at (\bstart-\seventh*1:\bigrad);
  \draw (\bstart-\seventh*1:\bigrad+\eps) node {$2$}; }

{ \coordinate (b2) at (\bstart-\seventh*2:\bigrad);
  \draw (\bstart-\seventh*2:\bigrad+\eps) node {$3$}; }

{ \coordinate (b3) at (\bstart-\seventh*3:\bigrad);
  \draw (\bstart-\seventh*3:\bigrad+\eps) node {$4$}; }

{ \coordinate (b4) at (\bstart-\seventh*4:\bigrad);
  \draw (\bstart-\seventh*4:\bigrad+\eps) node {$5$}; }

{ \coordinate (b5) at (\bstart-\seventh*5:\bigrad);
  \draw (\bstart-\seventh*5:\bigrad+\eps) node {$6$}; }

{ \coordinate (b6) at (\bstart-\seventh*6:\bigrad);
  \draw (\bstart-\seventh*6:\bigrad+\eps) node {$7$}; }

{ \coordinate (b7) at (\bstart-\seventh*7:\bigrad);
  \draw (\bstart-\seventh*7:\bigrad+\eps) node {$1$}; }

\foreach \n/\a/\r in {8/77/0.79, 10/130/0.5, 12/60/0.2, 14/350/0.5, 16/220/0.3, 18/225/0.75, 20/280/0.75,
    9/117/0.77, 11/92/0.38, 13/30/0.7, 15/290/0.05, 17/185/0.7, 19/250/0.55,  21/330/0.75}
{ \coordinate (b\n)  at (\a:\r*\bigrad); }


\foreach \e/\f/\t in {8/9/0.4, 9/10/0.5, 8/11/0.4, 10/11/0.5, 11/12/0.5, 8/13/0.5, 12/13/0.65, 13/14/0.4, 12/15/0.5,
 14/15/0.5, 15/16/0.5, 16/17/0.65, 10/17/0.6, 17/18/0.45, 18/19/0.5, 19/20/0.5, 20/21/0.5, 16/19/0.5, 14/21/0.5}
{\coordinate (a\e\f) at (${\t}*(b\e) + {1-\t}*(b\f)$); }


\draw [strand] plot[smooth]
coordinates {(b1) (a89) (a910) (a1017) (a1617) (a1619) (a1920) (b4)}
[ postaction=decorate, decoration={markings,
  mark= at position 0.16 with \strarrow, mark= at position 0.33 with \strarrow,
  mark= at position 0.46 with \strarrow, mark= at position 0.58 with \strarrow,
  mark= at position 0.69 with \strarrow, mark= at position 0.795 with \strarrow,
  mark= at position 0.94 with \strarrow }];
\draw [strand] plot[smooth] coordinates {(b2) (a1314) (a1415) (a1516) (a1617) (a1718) (b5)}
[ postaction=decorate, decoration={markings,
  mark= at position 0.13 with \strarrow, mark= at position 0.34 with \strarrow,
  mark= at position 0.50 with \strarrow, mark= at position 0.65 with \strarrow,
  mark= at position 0.79 with \strarrow, mark= at position 0.92 with \strarrow }];
\draw [strand] plot[smooth] coordinates {(b3) (a2021) (a1920) (a1819) (a1718) (b6)}
[ postaction=decorate, decoration={markings,
  mark= at position 0.09 with \strarrow, mark= at position 0.25 with \strarrow,
  mark= at position 0.40 with \strarrow, mark= at position 0.59 with \strarrow,
  mark= at position 0.83 with \strarrow }];
\draw [strand] plot[smooth] coordinates {(b4) (a2021) (a1421) (a1314) (a1213) (a1112) (a1011) (a910) (b7)}
[ postaction=decorate, decoration={markings,
  mark= at position 0.17 with \strarrow, mark= at position 0.335 with \strarrow,
  mark= at position 0.45 with \strarrow, mark= at position 0.56 with \strarrow,
  mark= at position 0.65 with \strarrow, mark= at position 0.73 with \strarrow,
  mark= at position 0.81 with \strarrow, mark= at position 0.93 with \strarrow, }];
\draw [strand] plot[smooth] coordinates {(b5) (a1819) (a1619) (a1516) (a1215) (a1213) (a813) (b1)}
[ postaction=decorate, decoration={markings,
  mark= at position 0.15 with \strarrow, mark= at position 0.33 with \strarrow,
  mark= at position 0.43 with \strarrow, mark= at position 0.55 with \strarrow,
  mark= at position 0.65 with \strarrow, mark= at position 0.78 with \strarrow,
  mark= at position 0.93 with \strarrow }];
  \draw [strand] plot[smooth] coordinates {(b6) (a1017) (a1011) (a811) (a813) (b2)}
[ postaction=decorate, decoration={markings,
  mark= at position 0.12 with \strarrow, mark= at position 0.30 with \strarrow,
  mark= at position 0.43 with \strarrow, mark= at position 0.56 with \strarrow,
  mark= at position 0.84 with \strarrow }];
\draw [strand] plot[smooth] coordinates {(b7) (a89) (a811) (a1112) (a1215) (a1415) (a1421) (b3)}
[ postaction=decorate, decoration={markings,
  mark= at position 0.08 with \strarrow, mark= at position 0.19 with \strarrow,
  mark= at position 0.32 with \strarrow, mark= at position 0.42 with \strarrow,
  mark= at position 0.53 with \strarrow, mark= at position 0.66 with \strarrow,
  mark= at position 0.88 with \strarrow }];


\foreach \n in {1,2,3,4,5,6,7} {\draw (b\n) circle(\bdrydotrad) [fill=gray];}


\foreach \n/\m in {1/671,2/712,3/123,4/234,5/345,6/456,7/567}
{ \draw (\qstart-\seventh*\n-25:\bigrad) node (q\m) {$\m$}; }

\foreach \m/\a/\r in {267/104/0.58 , 126/63/0.47, 256/160/0.25, 125/15/0.3, 356/220/0.52, 235/295/0.4 }
{ \draw (\a:\r*\bigrad) node (q\m) {$\m$}; }


\foreach \t/\h/\a in {671/267/-47,267/567/44, 267/126/18, 126/712/47, 125/126/-8, 123/125/-52, 126/256/3,
  256/267/-3, 256/125/-4, 256/356/-25, 567/256/-44, 356/456/31, 356/235/-2, 235/256/12, 125/235/2,
  235/345/35, 345/356/-54, 234/235/-62, 235/123/47}
{ \draw [quivarrow]  (q\t) edge [bend left=\a] (q\h); }


\foreach \t/\h/\a in {712/123/20, 123/234/20, 456/567/20, 567/671/20}
{ \draw [equivarrow]  (q\t) edge [bend left=\a] (q\h); }


\foreach \t/\h/\a in {712/671/-20, 345/234/-20,
  456/345/-20}
{ \draw [eequivarrow]  (q\t) edge [bend left=\a] (q\h); }

 \end{tikzpicture}
\]
\caption{The iced quiver $(\overline{Q}(D),F(D))$ of the Postnikov diagram $D$ in Figure~\ref{Example-Figure:iced quiver}, where the red arrows are the frozen arrows.}
\label{Example-Figure:iced quiver2}
\end{figure}
\end{example}

It has been proved in \cite{S06} that the coordinate ring $\C[\Gr(k,n)]$ has a cluster algebra structure, more precisely, the localization of $\C[\Gr(k,n)]$ at consecutive Pl\"{u}cker coordinates is isomorphic to the complexification of $\A_{(\overline{Q}(D),F)}$ as cluster algebras.

\begin{figure}
\[
\begin{tikzpicture}[scale=1.3,baseline=(bb.base),
  strand/.style={black,dashed,thick},
  quivarrow/.style={-latex, very thick}]
\newcommand{\strarrow}{\arrow{angle 60}}
\newcommand{\bdrydotrad}{0.07cm} 
\path (0,0) node (bb) {}; 

\draw [strand] plot[smooth]
coordinates {(0.6,0.8) (0.4,0.4) (0.2,0.15) (0,0)}
[ postaction=decorate, decoration={markings,
  mark= at position 0.5 with \strarrow}];
  \draw [strand] plot[smooth]
coordinates {(0.6,-0.8) (0.4,-0.4) (0.2,-0.15) (0,0)}
[ postaction=decorate, decoration={markings,
  mark= at position 0.5 with \strarrow}];
\draw [strand] plot[smooth]
coordinates {(0,0) (-0.2,0.15) (-0.4,0.4) (-0.6,0.8)}
[ postaction=decorate, decoration={markings,
  mark= at position 0.6 with \strarrow}];
\draw [strand] plot[smooth]
coordinates {(0,0) (-0.2,-0.15) (-0.4,-0.4) (-0.6,-0.8)}
[ postaction=decorate, decoration={markings,
  mark= at position 0.6 with \strarrow}];
\draw [quivarrow] (0.8,0)--(-0.8,0);
\draw[fill] (1,0) circle [radius=0.05];
\draw[fill] (-1,0) circle [radius=0.05];

\end{tikzpicture}
\qquad\qquad
\begin{tikzpicture}[scale=1.3,baseline=(bb.base),
  strand/.style={black,dashed,thick},
  quivarrow/.style={-latex, red, very thick}]
\newcommand{\strarrow}{\arrow{angle 60}}
\newcommand{\bdrydotrad}{0.07cm} 
\path (0,0) node (bb) {}; 

\draw [strand] plot[smooth]
coordinates {(0.6,0.8) (0.4,0.4) (0.2,0.15) (0,0)}
[ postaction=decorate, decoration={markings,
  mark= at position 0.5 with \strarrow}];
\draw [strand] plot[smooth]
coordinates {(0,0) (-0.2,0.15) (-0.4,0.4) (-0.6,0.8)}
[ postaction=decorate, decoration={markings,
  mark= at position 0.6 with \strarrow}];

\draw[] (0,0) circle [radius=0.05];
\draw (0,-0.2) node {\small $m$};
\draw [quivarrow] (.8,0)--(-.8,0);
\draw[fill] (1,0) circle [radius=0.05];
\draw[fill] (-1,0) circle [radius=0.05];
\end{tikzpicture}
%
%
\]
\caption{Arrow orientations in the quiver $\overline{Q}(D)$. The figures can also occur in the opposite sense, which means inverting the orientations of the strands and the arrows simultaneously.}\label{Figure: Orientation convention for the quiver}
\end{figure}

\begin{remark}
Note that the frozen arrows in the iced quiver $(\overline{Q}(D),F(D))$ have no influence on the ``cluster structure" of the cluster algebra ${\mathcal{A}}_{(\overline{Q}(D),F(D))}$. However, these arrows appear naturally in the (Frobenius) categorification of the cluster algebras in \cite{BKM16,GLS08,JKS16}. In particular, it is proved by Buar-King-Marsh in \cite{BKM16} that $(\overline{Q}(D),F(D))$ is the Gabriel quiver of the endomorphism algebra of the cluster tilting object in the Frobenius category (see
Section \ref{sec:application} for more information about this). These arrows are also important for the properties of the IQP itself (see Remarks \ref{rem:geo-ex-and-mut}, \ref{rem:nonrigidity} and \ref{rem:nonJacobi-finite}). The choice of such arrows ensures that we are able to define the IQP associated to the Grassmannian cluster algebras (see Remarks \ref{rem:geo-ex-and-mut} and Definition \ref{def:QP for Grassmannian cluster algebra}).
\end{remark}

\section{Quivers with potentials of Grassmannian cluster algebras}\label{Section:quivers with potentials of Grassmannian cluster algebras}
We introduce in this section two quivers with potentials $(\overline{Q}(D),F(D),\overline{W}(D))$ and $(Q(D),W(D))$ for each Postnikov diagram $D$, and verify the compatibility of geometric exchanges of $D$ and the mutations of these quivers with potentials. This ensures we may define quivers with potentials $(Q,W)$ and $(\overline{Q},F,\overline{W})$ for a Grassmannian cluster algebra up to mutation-equivalence. Then we prove the rigidity of $(Q, W)$ and the finiteness of the dimension of the corresponding Jacobian algebra $\mathcal{P}(Q, W)$. We also show that $(Q, W)$ is the unique rigid QP over the quiver $Q$ of the Grassmannian cluster algebras up to right-equivalence, and for each QP in the mutation class of $(Q, W)$,  we show that the quiver determines the potentials up to right-equivalence.

\subsection{The definition}\label{subsection:The definition}
We consider the boundary of $D$ as an oriented cycle, with clockwise orientation. We write $C.W.$ (resp. $A.C.W.$) for clockwise (resp. anticlockwise) for brevity. Then each oriented region $r$ in $D$ is either $C.W.$ or $A.C.W.$, and gives a set $[\omega_r]$ of cyclically equivalent cycles $\omega_r$ in the quiver whose arrows bound the region $r$, see Figure \ref{Figure: fundamental cycles}. Depending on the (internal or boundary) location of $r$, there are two kinds of such cycles, we call them the {\it internal fundamental cycles} and the {\it boundary fundamental cycles} respectively. We denote by $\mathfrak{R}(D)$ the set of all oriented regions in $D$.
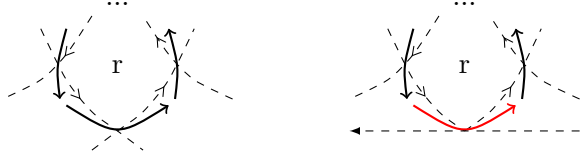
\begin{figure}
\[
\begin{tikzpicture}[scale=1.7,baseline=(bb.base),
  strand/.style={black,dashed},
  quivarrow/.style={black, -latex, very thick},
  boundaryarrow/.style={dashed, -latex}]
\newcommand{\strarrow}{\arrow{angle 60}}
\newcommand{\bdrydotrad}{0.03cm} 
\path (0,0) node (bb) {}; 

\draw [strand] plot[smooth]
coordinates {(-0.2,-0.15)(0,0)(0.2,0.15) (0.4,0.4)(0.6,0.8)}
[ postaction=decorate, decoration={markings,
  mark= at position 0.6 with \strarrow}];
\draw [strand] plot[smooth]
coordinates {(-0.6,0.8)(-0.4,0.4)(-0.2,0.15) (0,0)(0.2,-0.15)}
[ postaction=decorate, decoration={markings,
  mark= at position 0.5 with \strarrow}];
  \draw [strand] plot[smooth]
coordinates { (-0.2,0.9) (-0.4,0.6)(-0.6,0.4) (-0.9,0.25)}
[ postaction=decorate, decoration={markings,
  mark= at position -0.61 with \strarrow}];
    \draw [strand] plot[smooth]
coordinates { (0.9,0.25)(0.6,0.4) (0.4,0.6) (0.2,0.9) }
[ postaction=decorate, decoration={markings,
  mark= at position 0.8 with \strarrow}];
\draw[<-,thick] (-0.45,0.25)..controls(-0.48,0.5)..(-0.4,0.8);
\draw[->,thick] (0.45,0.25)..controls(0.48,0.5)..(0.4,0.8);
\draw[<-,thick] (0.4,0.2)..controls(0,-0.05)..(-0.4,0.2);

\draw (0,1) node {...};
\draw (0,0.5) node {r};

\end{tikzpicture}
\qquad\qquad
%
%
%
\begin{tikzpicture}[scale=1.7,baseline=(bb.base),
  strand/.style={black,dashed},
  quivarrow/.style={black, -latex, very thick},
  boundaryarrow/.style={dashed, -latex}]
\newcommand{\strarrow}{\arrow{angle 60}}
\newcommand{\bdrydotrad}{0.03cm} 
\path (0,0) node (bb) {}; 

\draw [strand] plot[smooth]
coordinates {(0,0) (0.2,0.15) (0.4,0.4) (0.6,0.8)}
[ postaction=decorate, decoration={markings,
  mark= at position 0.5 with \strarrow}];
\draw [strand] plot[smooth]
coordinates { (-0.6,0.8)(-0.4,0.4)(-0.2,0.15)(0,0) }
[ postaction=decorate, decoration={markings,
  mark= at position 0.6 with \strarrow}];
  \draw [strand] plot[smooth]
coordinates {(-0.2,0.9) (-0.4,0.6) (-0.6,0.4) (-0.9,0.25)}
[ postaction=decorate, decoration={markings,
  mark= at position -0.61 with \strarrow}];
    \draw [strand] plot[smooth]
coordinates {(0.9,0.25) (0.6,0.4) (0.4,0.6) (0.2,0.9)}
[ postaction=decorate, decoration={markings,
  mark= at position 0.8 with \strarrow}];
\draw[<-,thick] (-0.45,0.25)..controls(-0.48,0.5)..(-0.4,0.8);
\draw[->,thick] (0.45,0.25)..controls(0.48,0.5)..(0.4,0.8);
\draw[<-,red,thick] (0.4,0.2)..controls(0,-0.05)..(-0.4,0.2);
\draw [boundaryarrow] (0.9,0)--(-0.9,0);
\draw (0,1) node {...};
\draw (0,0.5) node {r};

\end{tikzpicture}
\]
\caption{The fundamental cycle, where the horizonal dashed line is the boundary of the diagram, and other dashed lines are the strands. The omitted part is at the interior of the diagram. The regions are anti-clockwise. When the regions are clockwise, the figures occur in the opposite sense, which means inverting the orientations of the strands and the arrows simultaneously. The first cycle is an internal fundamental cycle, while the second one is a boundary fundamental cycle.}
\label{Figure: fundamental cycles}
\end{figure}

\begin{definition}\label{Potentials for quivers of Postnikov diagrams}
For the quivers $Q(D)$ and $(\overline{Q}(D),F(D))$, let $W(D)$ and $\overline{W}(D)$ be potentials in the corresponding quivers, which are signed sums of representatives of fundamental cycles in the quivers, that is
\begin{equation*}
W=\sum\limits_{\substack{r\in\mathfrak{R}(D)~C.W.\\ \omega_{r}\in{Pot(Q(D))}}}{\omega_r}-\sum\limits_{\substack{r\in\mathfrak{R}(D)~A.C.W.\\ \omega_{r}\in{Pot(Q(D))}}}{\omega_r};~~~~~~~
\end{equation*}
\begin{equation*}
\overline{W}=\sum\limits_{\substack{r\in\mathfrak{R}(D)~C.W.\\ \omega_{r}\in{Pot(\overline{Q}(D))}}}{\omega_r}-\sum\limits_{\substack{r\in\mathfrak{R}(D)~A.C.W.\\ \omega_{r}\in{Pot(\overline{Q}(D))}}}{\omega_r}.
\end{equation*}
\end{definition}
Note that for each oriented region, we choose one representative in the potential, so these potentials are dependent on such choice. However, different choices are cyclically equivalent to each other. Clearly, there are no two cyclically equivalent cycles appearing in the potential simultaneously, and there is at least one unfrozen arrow in each term of the potential. So $(Q(D),W(D))$ is a QP and $(\overline{Q}(D),F(D),\overline{W}(D))$ is an IQP. Note that these (iced) QPs are also reduced by definition.


\begin{remark}\label{rem:BKM-IQP}
In a more general setting of dimer models, the potentials are always picked in such signed sum over the fundamental cycles. In particular, in the settings of Grassmannian cluster algebras, the IQP $(\overline{Q}(D),F(D),\overline{W}(D))$ has been introduced and studied in \cite{BKM16, P18}. The Jacobian algebra of the IQP is realized in \cite{BKM16} as the endomorphism algebra of the cluster tilting object in the JKS's Frobenius category.
\end{remark}

We are going to prove that  mutations of $(Q(D),W(D))$ and $(\overline{Q}(D),F,\overline{W}(D))$ are compatible with  geometric exchanges of the Postnikov diagrams $D$. We need the following lemma.
\begin{lemma}\label{Lemma:right-equivalence}
Let $D$ be a Postnikov diagram. let $\varepsilon: \{r\} \rightarrow \{\pm 1\}$ be a function on the set of oriented regions in $D$. Define
potentials
\begin{equation*}
W_\varepsilon=\sum\limits_{\substack{r\in\mathfrak{R}(D)~C.W.\\ \omega_{r}\in{Pot(Q(D))}}}\varepsilon(r){\omega_r}-\sum\limits_{\substack{r\in\mathfrak{R}(D)~A.C.W.\\ \omega_{r}\in{Pot(Q(D))}}}\varepsilon(r){\omega_r};
\end{equation*}
\begin{equation*}
\overline{W}_\varepsilon(D)=\sum\limits_{\substack{r\in\mathfrak{R}(D)~C.W.\\ \omega_{r}\in{Pot(\overline{Q}(D))}}}\varepsilon(r){\omega_r}-\sum\limits_{\substack{r\in\mathfrak{R}(D)~A.C.W.\\ \omega_{r}\in{Pot(\overline{Q}(D))}}}\varepsilon(r){\omega_r}.
\end{equation*}
Then $(Q(D),W_\varepsilon(D))$ and $(\overline{Q}(D),F(D),\overline{W}_\varepsilon(D))$) are right-equivalent to $(Q(D),W(D))$ and $(\overline{Q}(D),F(D),\overline{W}(D))$ respectively.
\end{lemma}
\begin{proof}
We only deal with the case of $Q(D)$, the case of $Q(D)$ is similar. Because the underlying graph of $Q(D)$ is a {\em planar graph} with {\em non-trivial} boundary, as stated for the QP arising from surfaces in section 10 of \cite{L16}, for any $\varepsilon,$ there exists a function
$$\xi: Q_1(D) \rightarrow \{\pm 1\}$$
on the arrows of $Q(D)$ such that for any $r$ with $\omega_r=\alpha_m \cdots\alpha_2 \alpha_1$, we have
$$\prod\limits_{i=1}^{m} \xi(\alpha_i)=\varepsilon(r).$$
 So the map
$$\phi:Q_1(D)\rightarrow Q_1(D), \alpha \mapsto \xi(\alpha)\alpha$$
 induces an algebra isomorphism $\Phi$ from $\C\langle\langle Q(D)\rangle\rangle$ to $\C\langle \langle Q(D)\rangle\rangle$ which maps $W(D)$ to $W_\varepsilon(D)$. Then $\Phi$ is a right-equivalence which completes the proof.
\end{proof}

Now we are ready to give the main result in this subsection.
\begin{Theorem}\label{compatible} The mutations of $(Q(D),W(D))$ and $(\overline{Q}(D),F,\overline{W}(D))$ are compatible with the geometric exchanges of the Postnikov diagram. More precisely, let $D$ be a reduced Postnikov diagram with an alternating oriented quadrilateral cell $R$, which associates to an exchangeable vertex $a$ in the quiver. Then up to right-equivalence, we have

(1) $\mu_{a}(Q(D))=Q(\mu_{R}(D))$ and $\mu_{a}(W(D))=W(\mu_{R}(D))$.

(2) ${\mu}_{a}(\overline{Q}(D),F(D))=(\overline{Q}(\mu_{R}(D)),F(\mu_{R}(D)))$ and $\mu_{a}(\overline{W}(D))=\overline{W}(\mu_{R}(D))$;

\end{Theorem}
\begin{proof}
Note that $(Q(D),W(D))$ is the ``principal part'' of $(\overline{Q}(D),F(D),\overline{W}(D))$, where $Q(D)$ is the principal part quiver of $(\overline{Q}(D),F(D))$ and $W(D)$ is obtained from $\overline{W}(D)$ by deleting the potentials which contain the frozen arrows. Thus the statement (1) follows from the statement (2). So we only prove the statement (2).

Since $R$ is quadrilateral in $D$, there are four arrows $\alpha,\beta,\gamma,\delta$ in $(\overline{Q}(D),F(D))$ whose endings are the associated vertex $a$. On the other hand, since $a$ is exchangeable, these arrows are all  unfrozen.
Without loss of generality, we assume that $a=s(\alpha)=t(\beta)$. Let $\omega_r=\alpha\beta p$ be a fundamental cycle in $\overline{W}(D)$ corresponding to an oriented region $r$ of $D$ that contains both $\alpha$ and $\beta$, where $p$ is a path from $t(\alpha)$ to $s(\beta)$.

Up to rotations and reflections, there are essentially three possibilities of $\omega_r$, which are shown in Figure \ref{Figure:theorem about IQP mutations and GE}.
If $length(p)=1$, then the local configuration of $D$ is as shown in picture $(1)$ and picture $(3)$, where the arrow $p$ is unfrozen in $(1)$ while it is frozen in $(3)$.
If $length(p)>1$, then the local configuration is as shown in picture $(2)$, where the length of the path $p$ is at least two and it may contain frozen arrows.
\begin{figure}
\begin{center}
%
\qquad
{\begin{tikzpicture}[scale=1.2]
\draw[fill] (-1,0) circle [radius=0.025];
\draw[fill] (0,0) circle [radius=0.025];
\draw[fill] (0,1) circle [radius=0.025];
\draw[] (-0.5,0.4) node {r};
\draw[->,thick] (-0.95,0) -- (-0.05,0);
\draw[->,thick] (0,0.1) -- (0,0.9);
\draw[->,thick] (-0.1,1) arc [radius=0.9, start angle=90, end angle=180];

\draw[->,thin,dashed] (-0.8,-0.3) -- (0.3,0.8);
\draw[thin,dashed] (-0.9,1) to [out=-80,in=135] (-0.5,0);
\draw[->,thin,dashed] (-0.5,0) to [out=-45,in=160] (0,-0.3);
\draw[->,thin,dashed] (0,0.5) to [out=160,in=20] (-1.3,0.5);
\draw[thin,dashed](0.3,0.35) to [out=155,in=-20](0,0.5);

\draw (-0.55,-0.3) node {$\beta$};
\draw (0.3,0.55) node {$\alpha$};
\draw (-1.1,0.8) node {$p$};
\draw (0.1,-0.1) node {$a$};

\draw (-0.5,-1) node {(1)};
\end{tikzpicture}}
\qquad
{\begin{tikzpicture}[scale=1.2]
\draw[fill] (-1,0) circle [radius=0.025];
\draw[fill] (0,0) circle [radius=0.025];
\draw[fill] (0,1) circle [radius=0.025];
\draw[] (-0.5,0.5) node {r};
\draw[->,thick] (-0.95,0) -- (-0.05,0);
\draw[->,thick] (0,0.1) -- (0,0.9);
\draw[->,thick,dashed] (-0.1,1) arc [radius=0.9, start angle=90, end angle=180];

\draw[->,thin,dashed] (-0.8,-0.3) -- (0.3,0.8);
\draw[<-,thin,dashed] (0,-0.5) -- (-1.3,0.8);
\draw[->,thin,dashed] (0.5,0) -- (-0.8,1.3);

\draw (-0.55,-0.3) node {$\beta$};
\draw (0.3,0.55) node {$\alpha$};
\draw (-0.9,0.8) node {$p$};
\draw (0.1,-0.1) node {$a$};

\draw (-0.5,-1) node {(2)};
\end{tikzpicture}}
\qquad
{\begin{tikzpicture}[scale=1.2]
\draw[fill] (-1,0) circle [radius=0.025];
\draw[fill] (0,0) circle [radius=0.025];
\draw[fill] (0,1) circle [radius=0.025];
\draw[] (-0.5,0.4) node {r};
\draw[->,thick] (-0.95,0) -- (-0.05,0);
\draw[->,thick] (0,0.1) -- (0,0.9);
\draw[->,red,thick] (-0.1,1) arc [radius=0.9, start angle=90, end angle=180];

\draw[->,thin,dashed] (-0.8,-0.3) -- (0.3,0.8);
\draw[thin,dashed] (-0.75,0.7) to [out=-80,in=135] (-0.5,0);
\draw[->,thin,dashed] (-0.5,0) to [out=-45,in=160] (0,-0.3);
\draw[->,thin,dashed] (0,0.5) to [out=150,in=0] (-0.75,0.7);
\draw[thin,dashed](0.3,0.3) to [out=140,in=-30](0,0.5);

\draw (-0.55,-0.3) node {$\beta$};
\draw (0.3,0.55) node {$\alpha$};
\draw (-0.9,0.8) node {$p$};
\draw (0.1,-0.1) node {$a$};

\draw[->,thin](-0.8,1) to [out=60,in=195](-0.5,1.2);
\draw (-0.5,-1) node {(3)};
\end{tikzpicture}}
\end{center}
\begin{center}
\caption{Local configuration of a fundamental circle through the vertex $a$. The reflections and the rotations are also allowed. The $p$ in the first picture is an unfrozen arrow, while the $p$ in the third picture is a frozen arrow. The $p$ in the second picture is a path with length at least two.}\label{Figure:theorem about IQP mutations and GE}
\end{center}
\end{figure}
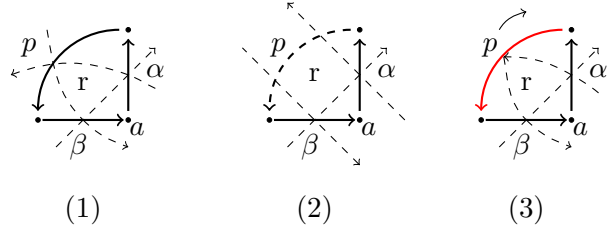

\begin{figure}
\begin{center}
{\begin{tikzpicture}[scale=1.75]
\draw[fill] (0,0) circle [radius=0.025];
\draw[fill] (-1,0) circle [radius=0.025];
\draw[fill] (0,1) circle [radius=0.025];
\draw[fill] (0,-1) circle [radius=0.025];
\draw[fill] (1,0) circle [radius=0.025];

\draw[->,thick] (-0.95,0) -- (-0.05,0);
\draw[->,thick] (0,0.05) -- (0,0.95);
\draw[->,thick] (0,-0.05) --(0,-0.95);
\draw[->,thick] (0.95,0) --(0.05, 0);
\draw[->,thick] (-0.05,-1) to [out=180,in=-90](-1,-0.05);
\draw[->,thick,dashed] (-1.05,-0.05)..controls(-1.25,-1.25)..(-0.05,-1.05);
\draw[<-,thick,dashed] (1.05,-0.05)..controls(1.25,-1.25)..(0.05,-1.05);

\draw[red,thin,dashed] (-0.5,0) -- (0,0.5);
\draw[red,thin,dashed] (-0.5,0) to [out=-135,in=90](-0.75,-0.7);
\draw[->,red,thin,dashed](0,0.5) to [out=45,in=180](0.75,0.7);
\draw[red,thin,dashed] (-0.75,-0.7)to(-0.75,-1) ;

\draw[thin,dashed](0,-0.5) to (0.5,0);
\draw[thin,dashed] (-0.75,-0.7) to [out=0,in=-135] (0,-0.5);
\draw[thin,dashed] (0.75,0.7) to [out=-90,in=45] (0.5,0);
\draw[->,thin,dashed] (-0.75,-0.7)to(-1,-0.7) ;

\draw[->,green,thin,dashed] (-0.5,0) to (0.5,-1);
\draw[green,thin,dashed] (-0.75,0.7) to [out=-90,in=135] (-0.5,0);

\draw[blue,thin,dashed](0,0.5) to (1,-0.5);
\draw[->,blue,thin,dashed] (0,0.5) to [out=135,in=0] (-0.75,0.7);

\draw (-0.5,-0.15) node {$\beta$};
\draw (-0.15,0.5) node {$\alpha$};
\draw (0.5,0.15) node {$\delta$};
\draw (0.15,-0.5) node {$\gamma$};
\draw (0.1,0.1) node {$a$};

\draw[->,red,thick] (0.05,1) to [out=0,in=90] (1,0.05);
\draw[->,red,thick] (-0.05,1) to [out=180,in=90] (-1,0.05);
\draw (0.85,0.8) node {$\zeta$};
\draw (-0.85,0.8) node {$\eta$};
\draw (-0.85,-0.8) node {$\xi$};
\draw (1,-1) node {$s$};
\draw (-1,-1) node {$t$};

\draw[->,thin](.5,1.2) to [out=-10,in=120](0.8,1);

\draw[->,thick] (1.5,0) to (3.5,0);
\node[] (C) at (2.5,0.2)
						{\mbox{pre-mutations}};
\node[] (C) at (2.5,-0.2)
						{\mbox{pre-geometric}};
\node[] (C) at (2.5,-0.4)
						{\mbox{exchanges}};
\node[] (C) at (0,-1.8)
						{$(D,\overline{Q}(D),F(D))$};
\end{tikzpicture}}
\qquad
{\begin{tikzpicture}[scale=1.75]
\draw[fill] (0,0) circle [radius=0.025];
\draw[fill] (-1,0) circle [radius=0.025];
\draw[fill] (0,1) circle [radius=0.025];
\draw[fill] (0,-1) circle [radius=0.025];
\draw[fill] (1,0) circle [radius=0.025];

\draw[<-,thick] (-0.95,0) -- (-0.05,0);
\draw[<-,thick] (0,0.05) -- (0,0.95);
\draw[<-,thick] (0,-0.05) --(0,-0.95);
\draw[<-,thick] (0.95,0) --(0.05, 0);

\draw[->,thick] (-0.95,0.05) -- (-0.05,0.95);
\draw[->,thick] (0.95,0.05) -- (0.05,0.95);
\draw[->,thick] (-0.95,-0.05) -- (-0.05,-0.95);
\draw[->,thick] (0.95,-0.05) -- (0.05,-0.95);

\draw[->,thick] (0.05,1) to [out=0,in=90] (1,0.05);
\draw[->,thick] (-0.05,-1) to [out=180,in=-90](-1,-0.05);

\draw[->,thick,dashed] (-1.05,-0.05)..controls(-1.25,-1.25)..(-0.05,-1.05);
\draw[<-,thick,dashed] (1.05,-0.05)..controls(1.25,-1.25)..(0.05,-1.05);

\draw[red,thin,dashed] (-0.75,-0.7)to(-0.75,-1);
\draw[red,dashed,thin] (-0.75,-0.7) to [out=90,in=180](-0.5,-0.5);
\draw[red,dashed,thin] (0.5,0) to [out=-90,in=0](0,-0.5);
\draw[red,dashed] (0.5,0) to(0.5,0.5);
\draw[red,dashed] (-0.5,-0.5) to(0,-0.5);
\draw[->,red,dashed,thin] (0.5,0.5) to [out=90,in=180](0.75,0.7);

\draw[->,thin,dashed] (-0.75,-0.7)to(-1,-0.7) ;
\draw[dashed,thin] (-0.75,-0.7) to [out=0,in=-90](-0.5,-0.5);
\draw[dashed,thin](0.5,0.5) to [out=0,in=-90](0.75,0.7);
\draw[dashed,thin] (-0.5,0) to [out=90,in=180](0,0.5);
\draw[dashed] (0.5,0.5) to(0,0.5);
\draw[dashed] (-0.5,-0.5) to(-0.5,0);

\draw[blue,dashed,thin] (-0.5,0) to [out=-90,in=180](0,-0.5);
\draw[blue,dashed] (-0.5,0) to(-0.5,0.5);
\draw[blue,dashed] (0.5,-0.5) to(0,-0.5);
\draw[blue,dashed] (0.5,-0.5) to(0.8,-0.5);
\draw[->,blue,dashed,thin] (-0.5,0.5) to [out=90,in=0](-0.75,0.7);

\draw[green,dashed,thin] (0.5,0) to [out=90,in=0](0,0.5);
\draw[green,dashed] (-0.5,0.5) to(0,0.5);
\draw[green,dashed] (0.5,0) to(0.5,-0.5);
\draw[->,green,dashed] (0.5,-0.5) to(0.5,-0.8);
\draw[green,dashed,thin] (-0.5,0.5) to [out=180,in=270](-0.75,0.7);

\draw (-0.15,0.3) node {$\alpha^*$};
\draw (-0.3,-0.15) node {$\beta^*$};
\draw (0.15,-0.3) node {$\gamma^*$};
\draw (0.3,0.15) node {$\delta^*$};
\draw (0.78,0.38) node {{\bahao $[\alpha\delta]$}};
\draw (0.78,-0.38) node {{\bahao $[\gamma\delta]$}};
\draw (-0.78,0.38) node {{\bahao $[\alpha\beta]$}};
\draw (-0.78,-0.38) node {{\bahao $[\gamma\beta]$}};

\draw[->,red,thick] (0.05,1) to [out=0,in=90] (1,0.05);
\draw[->,red,thick] (-0.05,1) to [out=180,in=90] (-1,0.05);
\draw (0.85,0.8) node {$\zeta$};
\draw (-0.85,0.8) node {$\eta$};
\draw (-0.85,-0.8) node {$\xi$};
\draw (1,-1) node {$s$};
\draw (-1,-1) node {$t$};

\draw[->,thin](.5,1.2) to [out=-10,in=120](0.8,1);


\node[] (C) at (0,-1.8)
						{$(\widetilde{\mu}_R(D),\widetilde{\mu}_{a}(\overline{Q}(D)),\widetilde{\mu}_{a}(F(D)))$};
\end{tikzpicture}}
\qquad\\
\qquad\\
{\begin{tikzpicture}[scale=1.8]
\draw[fill] (0,0) circle [radius=0.025];
\draw[fill] (-1,0) circle [radius=0.025];
\draw[fill] (0,1) circle [radius=0.025];
\draw[fill] (0,-1) circle [radius=0.025];
\draw[fill] (1,0) circle [radius=0.025];

\draw[<-,thick] (-0.95,0) -- (-0.05,0);
\draw[<-,thick] (0,0.05) -- (0,0.95);
\draw[<-,thick] (0,-0.05) --(0,-0.95);
\draw[<-,thick] (0.95,0) --(0.05, 0);

\draw[->,thick,dashed] (-1.05,-0.05)..controls(-1.25,-1.25)..(-0.05,-1.05);
\draw[<-,thick,dashed] (1.05,-0.05)..controls(1.25,-1.25)..(0.05,-1.05);

\draw[<-,red,thick] (0.05,1) to [out=0,in=90] (1,0.05);
\draw[<-,red,thick] (-0.05,1) to [out=180,in=90] (-1,0.05);
\draw[->,thick](1,-0.05) to [out=-90,in=0] (0.05,-1);

\draw[red,dashed] (-0.3,-0.8) to(0.5,0);
\draw[<-,red,thin,dashed] (0.75,0.7) to [out=-90,in=45] (0.5,0);

\draw[<-,dashed] (-0.8,-0.3) to(0,0.5);
\draw[thin,dashed](0,0.5) to [out=45,in=180](0.75,0.7);

\draw[green,thin,dashed](-0.75,0.7) to [out=0,in=135] (0,0.5);
\draw[green,thin,dashed] (0.5,0) -- (0,0.5);
\draw[green,thin,dashed] (0.5,0) to [out=-45,in=90](0.75,-0.7);
\draw[->,green,thin,dashed] (0.75,-0.7)to(0.75,-1) ;

\draw[blue,thin,dashed] (1,-0.7) to (0.75,-0.7);
\draw[blue,thin,dashed](0.75,-0.7) to [out=180,in=-45] (0,-0.5);
\draw[blue,thin,dashed](0,-0.5) to (-0.5,0);
\draw[->,blue,thin,dashed] (-0.5,0) to [out=135,in=-90] (-0.75,0.7);

\draw (-0.5,-0.2) node {$\beta^*$};
\draw (-0.17,0.5) node {$\alpha^*$};
\draw (0.5,0.18) node {$\delta^*$};
\draw (0.18,-0.5) node {$\gamma^*$};
\draw (0.9,-0.85) node {{\bahao $[\gamma\delta]$}};
\draw (0.95,0.75) node {{\bahao $[\alpha\delta]$}};
\draw (-0.85,0.8) node {{\bahao $[\alpha\beta]$}};
\draw (1,-1) node {$s$};
\draw (-1,-1) node {$t$};

\draw[->,thin](.5,1.2) to [out=-10,in=120](0.8,1);

\draw[<-,thick] (-1.5,0) to (-3.5,0);
\node[] (C) at (-2.5,0.4)
						{\mbox{reduction:}};
\node[] (C) at (-2.5,0.2)
						{\mbox{$2$-cycles moves}};
\node[] (C) at (-2.5,-0.2)
						{\mbox{untwisting moves}};
\node[] (C) at (0,-1.8)
						{$({\mu}_R(D),\mu_{a}(\overline{Q}(D)),\mu_{a}(F(D))$};
\end{tikzpicture}}
\end{center}
\begin{center}
\caption{Mutations and geometric exchanges}\label{Figure:theorem about IQP mutations and GE 2}
\end{center}
\end{figure}
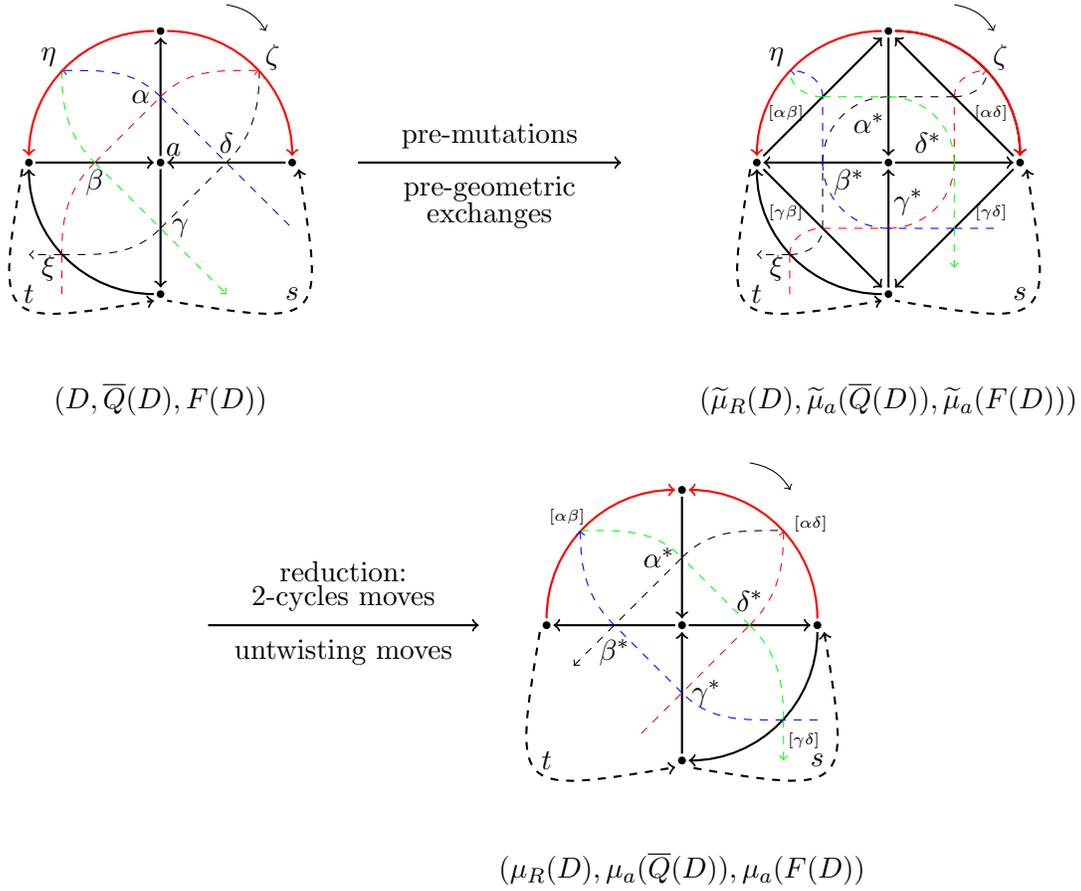

Now we consider the local configuration of $D$ around $a$ depicted in the first picture of Figure \ref{Figure:theorem about IQP mutations and GE 2}, which contains all the above three possibilities. We only prove the result for this situation. Other situations can be proved similarly.

Up to a cyclical equivalence, we may write the potential
$$\overline{W}(D)=\zeta\alpha \delta-s\gamma\delta + \xi\gamma \beta-\eta\alpha\beta-\xi t+\overline{W}'(D) ,$$
where $length(s)\geqslant 2$, $length(t)\geqslant 2$ and each cycle in $\overline{W}'(D)$  does not contain any arrow in $\alpha,\beta,\gamma,\delta,\xi,\zeta$, and $\eta$. Then by a pre-geometric exchange $\widetilde{\mu}_R$ on $D$ and a pre-mutation $\widetilde{\mu}_a$ on $(\overline{Q}(D),F(D))$, we obtain the second picture in Figure \ref{Figure:theorem about IQP mutations and GE 2}.
Note that the new arrows appearing in the iced quiver $\widetilde{\mu}_{a}(\overline{Q}(D)),\widetilde{\mu}_{a}(F(D))$ are all unfrozen. Meanwhile, by applying the pre-mutation $\widetilde{\mu}_a$ on $\overline{W}(D)$, we get a new potential
$$\begin{array}{rcl}
				\widetilde{\mu}_a(\overline{W}(D)) & = & \zeta[\alpha\delta]+\delta^*\alpha^*[\alpha\delta]-s[\gamma\delta]-\delta^*\gamma^*[\gamma\delta]
+\xi[\gamma\beta]\\
				    &  & +\beta^*\gamma^*[\gamma\beta]-\eta[\alpha\beta]-\beta^*\alpha^*[\alpha\beta]
-\xi t+\overline{W}'(D).
			\end{array}$$
Note that $\overline{W}'(D)$ is not changed because any arrow appearing in each potential of $\overline{W}'(D)$ is not involved in $a$.

Recall the processes of the reduction of an IQP stated in Section \ref{subsec:QP}, to reduce the IQP $(\widetilde{\mu}_a(\overline{Q}(D)),\widetilde{\mu}_{a}(F(D)),\widetilde{\mu}_a(\overline{W}(D)))$, we should firstly find a right-equivalence and use it to rewrite the potential as the canonical form \eqref{eq:potential}. Let us consider a unitriangular automorphism $\phi$ on $\C\langle\langle\widetilde{\mu}_a(\overline{Q}(D))\rangle\rangle$, where
$$\phi([\gamma\beta])=[\gamma\beta]+t,\phi(\xi)=\xi-\beta^*\gamma^*,\phi([\alpha\beta])=-[\alpha\beta],\phi(u)=u$$
for other arrows $u$ in $\widetilde{\mu}_a((\overline{Q}(D))$.
Then
$$\begin{array}{rcl}
				\phi(\widetilde{\mu}_a(\overline{W}(D))) & = & \xi[\gamma\beta]\\
				 &   & +[\alpha\delta](\zeta+\delta^*\alpha^*)+[\alpha\beta]
(\eta+\beta^*\alpha^*)\\
				 &   & +\beta^*\gamma^*t-s[\gamma\delta]+\delta^*\gamma^*[\gamma\delta]+\overline{W}'(D).
			\end{array}$$
On the one hand, note that $\phi$ gives a right equivalence between the IQPs
$$(\widetilde{\mu}_a(\overline{Q}(D)),\widetilde{\mu}_a(F(D)),\widetilde{\mu}_a(\overline{W}(D)))$$
and $$(\widetilde{\mu}_a(\overline{Q}(D)),\widetilde{\mu}_a((F(D)),\phi(\widetilde{\mu}_a(\overline{W}(D)))),$$ in particular, $$\phi(\C\langle \langle F(D)\rangle\rangle)=\C\langle \langle \tilde{\mu}_a(F(D))\rangle\rangle.$$
On the other hand, $\phi(\widetilde{\mu}_a(\overline{W}(D)))$ is of the canonical form \eqref{eq:potential} with the reduced part
$$\phi(\widetilde{\mu}_a(\overline{W}(D)))_\mathrm{red}=[\alpha\delta]\delta^*\alpha^*+[\alpha\beta]\beta^*\alpha^*
+\beta^*\gamma^*t-s[\gamma\delta]+\delta^*\gamma^*[\gamma\delta]+\overline{W}'(D).$$

Then after the reduction, we get the mutation
$$(\mu_a(\overline{Q}(D)),\mu_a(F(D)),\mu_a(\overline{W}(D))),$$
where the iced quiver $(\mu_a(\overline{Q}(D)),\mu_a(F(D)))$ is obtained from $(\widetilde{\mu}_a(\overline{Q}(D)),\widetilde{\mu}_a(F(D)))$
by deleting the arrows $\xi$, $[\gamma\beta]$, $\zeta$ and $\eta$, and freezing the arrows $[\alpha\beta]$ and $[\alpha\delta]$.
Note that this iced quiver is exactly the iced quiver of the final Postnikov diagram $\mu_{R}(D)$ depicted in the third picture of Figure \ref{Figure:theorem about IQP mutations and GE 2}, that is,
we have $\overline{Q}(\mu_{R}(D))=\mu_a(\overline{Q}(D))$.

On the other hand, by the definition,
$$\overline{W}(\mu_R(D))=-[\alpha\delta]\delta^*\alpha^*+[\alpha\beta]\beta^*\alpha^*
-\beta^*\gamma^*t-s[\gamma\delta]+\delta^*\gamma^*[\gamma\delta]+\overline{W}'(D).$$
And by the Lemma \ref{Lemma:right-equivalence}, there is a sign changing of arrows $\psi$ on $\C\langle\langle\mu_a(\overline{Q}(D))\rangle\rangle$ such that $$\psi(\mu_a(\overline{W}(D)))=\psi(\phi(\widetilde{\mu}_a(\overline{W}(D)))_{red})=\overline{W}(\mu_{R}(D))$$
up to the equality $\mu_a(\overline{Q}(D))=\overline{Q}(\mu_{R}(D))$.
So by the right-equivalent $\psi \phi$, we obtain the final mutation $\mu_{a}(\overline{Q}(D),F(D),\overline{W}(D))$, as well as the wanted equalities $$(\mu_{a}(\overline{Q}(D)),\mu_{a}(F(D))=(\overline{Q}(\mu_{R}(D)),F(\mu_{R}(D))) \mbox{~and~} \mu_{a}(\overline{W}(D))=\overline{W}(\mu_{R}(D)).$$
\end{proof}

The compatibility stated in the above theorem ensures the following definition.
\begin{definition}\label{def:QP for Grassmannian cluster algebra}
Let $D$ be any $(k,n)$-Postnikov diagram and let $\C[Gr(k,n)]$ be the Grassmannian cluster algebra. We call
\begin{itemize}
\item a QP which is mutation equivalent to $(Q(D),W(D))$ a QP of $\C[Gr(k,n)]$, and denote it by $(Q,W)$;
\item an IQP which is mutation equivalent to $(\overline{Q}(D),F(D),\overline{W}(D))$ an IQP of $\C[Gr(k,n)]$, and denote it by $(\overline{Q},F,\overline{W})$.
\end{itemize}
\end{definition}

\subsection{Rigidity and finite dimension}\label{sec:rigidity}
We prove in this subsection that each QP of a Grassmannian cluster algebra $\C[Gr(k,n)]$ is rigid and Jacobi-finite.
Recall that a QP $(Q,W)$ is said to be $2$-acyclic if there are no $2$-cycles in the quiver. Note that there may appear $2$-cycles in the quiver of $\mu_i(Q,W)$ after mutations, even if $(Q,W)$ is $2$-acyclic. If all possible iterations of mutations are $2$-acyclic, then we say $(Q,W)$ is {\em non-degenerate}. We call $(Q,W)$ {\em rigid} if every cycle in $Q$ is cyclically equivalent to an element of the Jacobian ideal $J(Q,W)$. It is known that a rigid QP is always non-degenerate. We call $(Q,W)$ {\em Jacobi-finite} if the Jacobian algebra $\mathcal{P}(Q,W)$ is finite dimensional.

For the further study, we need some special Postnikov diagram, see Figure \ref{Figure:initial Postnikov diagram}, where the diagrams depend on the parities of $k$ and $n$, and any pair $(k,n)$ matches a unique diagram shown in these figures. These diagrams are of special importance, they are used by Scott as the initial diagrams which give the initial quivers of the Grassmannian cluster algebras \cite{S06}.

Denote by $(\overline{Q}_{\textrm{ini}},F_{\textrm{ini}},\overline{W}_{\textrm{ini}})$ and $(Q_{\textrm{ini}},W_{\textrm{ini}})$ the IQP and the QP associated to the initial Postnikov diagram respectively. In the following we always assume that $k$ and $n$ are both odd. The statements and the proofs for the other cases are similar. The quiver $Q$ is a kind of grid as showed in Figure \ref{Figure:initial quiver1}, where we endow the points of the quiver with \emph{coordinates}, and label the position of a fundamental cycle by its row Rj and column Ci. For example, the left bottom fundamental cycle lies at R1 row and C1 column. We denote by $a(i,j)$ the vertex with coordinate $(i,j)$.

Let $l$ be a path which forms a cycle in $Q_{\textrm{ini}}$. Let $a(i_1,j_1)$ be a vertex on $l$, we say it is a \emph{leftmost vertex} of $l$ if $i_1\leqslant i$ for any vertex $a(i,j)$ on $l$. Similarly, we define the \emph{rightmost vertex}, \emph{lowest vertex} and \emph{highest vertex} of $l$ as $a(i_2,j_2)$, $a(i_3,j_3)$ and $a(i_4,j_4)$ respectively. We call $width(l)=i_2-i_1$ the {\em width} of $l$, and call $height(l)=j_4-j_3$ the {\em height} of $l$.

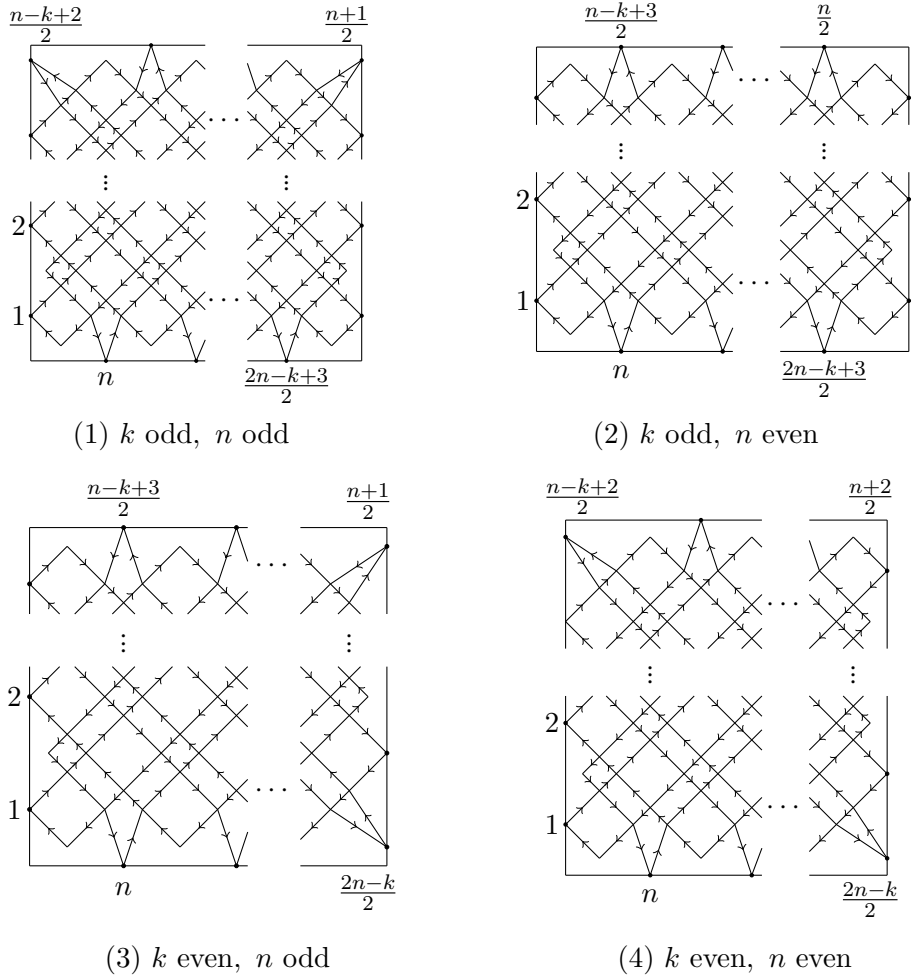
\begin{figure}
\begin{center}
{\begin{tikzpicture}[scale=0.04]

\foreach \x in {0,30cm,90cm}
\foreach \y in {0,30cm,60cm}
\draw[->,xshift=\x,yshift=\y](4.9,24.9) -- (5,25);

\foreach \x in {0,30cm,60cm}
\foreach \y in {0,60cm}
\draw[->,xshift=\x,yshift=\y](19.9,39.9) -- (20,40);

\foreach \x in {0,30cm}
\foreach \y in {0,60cm}
\draw[->,xshift=\x,yshift=\y](11.5,31.5) -- (12.5,32.5);
\draw[->,xshift=90cm](11.5,31.5) -- (12.5,32.5);

\foreach \x in {0,60cm}
\foreach \y in {0,30cm}
\draw[->,xshift=\x,yshift=\y](27.4,47.4) -- (27.5,47.5);

\foreach \x in {0,60cm}
\foreach \y in {0}
\draw[->,xshift=\x,yshift=\y](28,14.2) -- (28.1,14.5);

\foreach \x in {0,30cm,60cm,90cm}
\foreach \y in {0,30cm,60cm}
\draw[->,xshift=\x,yshift=\y](14.9,25.1) -- (15,25);

\foreach \x in {0,30cm,90cm}
\foreach \y in {0}
\draw[->,xshift=\x,yshift=\y](7.4,32.6) -- (7.5,32.5);

\foreach \x in {0,60cm}
\foreach \y in {0,60cm}
\draw[->,xshift=\x,yshift=\y](29.9,40.1) -- (30,40);

\foreach \x in {0,60cm}
\foreach \y in {0,30cm,60cm}
\draw[->,xshift=\x,yshift=\y](27.6,22.4) -- (27.5,22.5);

\foreach \x in {0,30cm,60cm}
\foreach \y in {0,-30cm}
\draw[->,xshift=\x,yshift=\y](22.4,77.6) -- (22.5,77.5);

\foreach \x in {0,60cm}
\foreach \y in {0}
\draw[->,xshift=\x,yshift=\y](37.4,92.6) -- (37.5,92.5);
\foreach \x in {0,30cm,90cm}
\foreach \y in {0,30cm,60cm}
\draw[->,xshift=\x,yshift=\y](5.1,14.9) -- (5,15);

\foreach \x in {0,30cm,60cm}
\foreach \y in {0,60cm}
\draw[->,xshift=\x,yshift=\y](20.1,29.9) -- (20,30);

\foreach \x in {0,30cm,90cm}
\foreach \y in {0}
\draw[->,xshift=\x,yshift=\y](12.6,37.4) -- (12.5,37.5);

\foreach \x in {0,30cm,60cm,90cm}
\foreach \y in {0,30cm,60cm}
\draw[<-,xshift=\x,yshift=\y](15,15) -- (15.1,15.1);

\foreach \x in {0,30cm,60cm}
\foreach \y in {0,30cm,60cm}
\draw[<-,xshift=\x,yshift=\y](22.5,22.5) -- (22.6,22.6);

\foreach \x in {0,60cm}
\foreach \y in {0,60cm}
\draw[->,xshift=\x,yshift=\y](30.1,30.1) -- (30,30);

\foreach \x in {0,30cm,90cm}
\foreach \y in {0}
\draw[->,xshift=\x,yshift=\y](7.6,37.6) -- (7.5,37.5);
\draw[<-] (37,101)--(37.1,101.3);

\foreach \x in {0,30cm,60cm}
\draw[<-,xshift=\x] (22.5,12.5)--(22.4,12.8);

\foreach \x in {0}
\draw[<-,xshift=\x] (42.5,102.5)--(42.6,102.2);

\draw[->] (7.7,99.9)--(7.5,100);
\draw[->] (5.25,97.2)--(5.4,97);

\draw[<-] (102.2,99.8)--(102.5,100);
\draw[->] (105.3,97.9)--(105.4,98);

\draw[] (0,5)--(0,58);
\draw[] (0,72)--(0,110);
\draw[] (0,5)--(58,5);
\draw[] (72,5)--(110,5);
\draw[] (0,110)--(58,110);
\draw[] (72,110)--(110,110);
\draw[] (110,5)--(110,58);
\draw[] (110,72)--(110,110);
\node[] (C) at (65,25)
						{$\cdots$};
\node[] (C) at (65,85)
						{$\cdots$};
\node[] (C) at (25,62)
						{$\cdot$};
\node[] (C) at (25,64)
						{$\cdot$};
\node[] (C) at (25,66)
						{$\cdot$};
\node[] (C) at (85,62)
						{$\cdot$};
\node[] (C) at (85,64)
						{$\cdot$};
\node[] (C) at (85,66)
						{$\cdot$};

\draw[] (0,20)--(38,58);
\draw[] (52,72)--(58,78);
\draw[] (72,92)--(85,105);
\draw[] (85,105)--(110,80);

\node[] (C) at (-4,20)
						{$1$};
\draw[] (110,80)--(102,72);
\draw[] (88,58)--(72,42);
\draw[] (58,28)--(40,10);
\draw[] (40,10)--(0,50);
\draw[] (0,50)--(8,58);
\draw[] (22,72)--(55,105);
\draw[] (55,105)--(58,102);
\draw[] (72,88)--(88,72);
\draw[] (102,58)--(110,50);
\draw[fill] (0,20) circle [radius=0.5];
\draw[fill] (0,50) circle [radius=0.5];
\draw[fill] (0,80) circle [radius=0.5];
\draw[fill]  (110,50)circle [radius=0.5];
\draw[fill]  (110,80)circle [radius=0.5];
\draw[fill]  (40,110)circle [radius=0.5];
\draw[fill] (110,105) circle [radius=0.5];
\draw[fill] (110,20) circle [radius=0.5];
\draw[fill] (0,105) circle [radius=0.5];
\draw[fill] (55,5) circle [radius=0.5];
\draw[fill] (85,5) circle [radius=0.5];
\draw[fill] (25,5) circle [radius=0.5];

\node[] (C) at (-4,50)
						{$2$};
\draw[] (110,50)--(72,12);
\draw[] (58,22)--(22,58);
\draw[] (8,72)--(0,80);
\draw[] (0,80)--(25,105);
\draw[](25,105)--(58,72);
\draw[](72,58)--(110,20);
\draw[](110,20)--(100,10);
\draw[](100,10)--(72,38);
\draw[](58,52)--(52,58);
\draw[](38,72)--(15,95);
\draw[](15,95)--(0,105);

\node[] (C) at (85,-3)
						{$\frac{2n-k+3}{2}$};
\draw[](0,105)--(10,90);
\draw[](10,90)--(28,72);
\draw[](42,58)--(58,42);
\draw[](72,28)--(80,20);
\draw[](80,20)--(85,5);

\node[] (C) at (5,117)
						{$\frac{n-k+2}{2}$};
\draw[](85,5)--(90,20);
\draw[](90,20)--(105,35);
\draw[](105,35)--(82,58);
\draw[](45,95)--(58,82);
\draw[](45,95)--(40,110);
\draw[](40,110)--(35,95);
\draw[](35,95)--(12,72);
\draw[](12,58)--(50,20);
\draw[](50,20)--(55,5);
\draw[](55,5)--(58,12);
\draw[](72,32)--(98,58);
\draw[](98,72)--(75,95);
\draw[](75,95)--(72,104);
\draw[](58,88)--(42,72);
\draw[](28,58)--(5,35);
\draw[](5,35)--(20,20);
\draw[](20,20)--(25,5);

\node[] (C) at (25,-1)
						{$n$};
\draw[](25,5)--(30,20);
\draw[](30,20)--(58,48);
\draw[](82,72)--(100,90);
\draw[](100,90)--(110,105);
\draw[](110,105)--(95,95);
\draw[](95,95)--(72,72);
\draw[](58,58)--(10,10);
\draw[](10,10)--(0,20);

\node[] (C) at (105,117)
						{$\frac{n+1}{2}$};
\node[] (C) at (50,-20)
						{$(1)~k~\mbox{odd},~n~\mbox{odd}$};
\end{tikzpicture}}
\qquad\qquad
{\begin{tikzpicture}[scale=0.045]


\foreach \x in {0,30cm,90cm}
\foreach \y in {0,30cm,60cm}
\draw[->,xshift=\x,yshift=\y](4.9,24.9) -- (5,25);

\foreach \x in {0,30cm,60cm}
\foreach \y in {0}
\draw[->,xshift=\x,yshift=\y](19.9,39.9) -- (20,40);

\foreach \x in {0,30cm}
\foreach \y in {0}
\draw[->,xshift=\x,yshift=\y](11.5,31.5) -- (12.5,32.5);
\draw[->,xshift=90cm](11.5,31.5) -- (12.5,32.5);

\foreach \x in {0,60cm}
\foreach \y in {0,30cm}
\draw[->,xshift=\x,yshift=\y](27.4,47.4) -- (27.5,47.5);

\foreach \x in {0,60cm}
\foreach \y in {0}
\draw[->,xshift=\x,yshift=\y](27.4,12.2) -- (27.5,12.5);

\foreach \x in {0,30cm,60cm,90cm}
\foreach \y in {0,30cm,60cm}
\draw[->,xshift=\x,yshift=\y](14.9,25.1) -- (15,25);

\foreach \x in {0,30cm,90cm}
\foreach \y in {0}
\draw[->,xshift=\x,yshift=\y](7.4,32.6) -- (7.5,32.5);

\foreach \x in {0,60cm}
\foreach \y in {0}
\draw[->,xshift=\x,yshift=\y](29.9,40.1) -- (30,40);

\foreach \x in {0,30cm,60cm}
\foreach \y in {0,-30cm}
\draw[->,xshift=\x,yshift=\y](22.4,77.6) -- (22.5,77.5);

\foreach \x in {0,30cm,90cm}
\foreach \y in {0,30cm,60cm}
\draw[->,xshift=\x,yshift=\y](5.1,14.9) -- (5,15);

\foreach \x in {0,30cm,60cm}
\foreach \y in {0}
\draw[->,xshift=\x,yshift=\y](20.1,29.9) -- (20,30);

\foreach \x in {0,30cm,90cm}
\foreach \y in {0}
\draw[->,xshift=\x,yshift=\y](12.6,37.4) -- (12.5,37.5);

\foreach \x in {0,60cm}
\foreach \y in {0,30cm}
\draw[->,xshift=\x,yshift=\y](27.6,22.4) -- (27.5,22.5);

\foreach \x in {0,30cm,60cm,90cm}
\foreach \y in {0,30cm,60cm}
\draw[<-,xshift=\x,yshift=\y](15,15) -- (15.1,15.1);

\foreach \x in {0,30cm,60cm}
\foreach \y in {0,30cm}
\draw[<-,xshift=\x,yshift=\y](22.5,22.5) -- (22.6,22.6);

\foreach \x in {0,60cm}
\foreach \y in {0}
\draw[->,xshift=\x,yshift=\y](30.1,30.1) -- (30,30);

\foreach \x in {0,30cm,90cm}
\foreach \y in {0}
\draw[->,xshift=\x,yshift=\y](7.6,37.6) -- (7.5,37.5);

\foreach \x in {0,30cm,60cm}
\draw[<-,xshift=\x] (22.8,11.9)--(22.4,12.8);

\foreach \x in {0,60cm}
\draw[->,xshift=\x] (27.5,87.5)--(27.3,88);

\foreach \x in {0,30cm,60cm}
\draw[<-,xshift=\x] (22.1,86.9)--(22.5,87.8);

\draw[] (0,5)--(0,58);
\draw[] (0,72)--(0,95);
\draw[] (0,5)--(58,5);
\draw[] (72,5)--(110,5);
\draw[] (0,95)--(58,95);
\draw[] (72,95)--(110,95);
\draw[] (110,5)--(110,58);
\draw[] (110,72)--(110,95);

\draw[] (0,20)--(38,58);
\draw[] (52,72)--(58,78);
\draw[] (72,88)--(88,72);
\draw[] (102,58)--(110,50);
\node[] (C) at (-4,20)
						{$1$};
\draw[] (110,50)--(72,12);
\draw[] (58,22)--(22,58);
\draw[] (8,72)--(0,80);
\draw[] (0,80)--(10,90);
\draw[] (10,90)--(28,72);
\draw[] (42,58)--(58,42);
\draw[] (72,28)--(80,20);
\draw[] (80,20)--(85,5);
\node[] (C) at (25,103)
						{$\frac{n-k+3}{2}$};
\draw[] (85,5)--(90,20);
\draw[] (90,20)--(105,35);
\draw[](105,35)--(82,58);
\draw[](58,88)--(55,95);
\draw[](55,95)--(50,80);
\draw[](28,58)--(5,35);
\draw[](50,80)--(42,72);
\draw[](5,35)--(20,20);
\draw[](20,20)--(25,5);
\draw[](25,5)--(30,20);
\draw[](30,20)--(58,48);
\draw[](82,72)--(100,90);
\draw[](100,90)--(110,80);

\draw[fill] (0,20) circle [radius=0.5];
\draw[fill] (0,50) circle [radius=0.5];
\draw[fill] (0,80) circle [radius=0.5];
\draw[fill]  (110,50)circle [radius=0.5];
\draw[fill]  (110,80)circle [radius=0.5];
\draw[fill] (85,95) circle [radius=0.5];
\draw[fill] (55,95) circle [radius=0.5];
\draw[fill] (110,20) circle [radius=0.5];
\draw[fill] (25,95) circle [radius=0.5];
\draw[fill] (55,5) circle [radius=0.5];
\draw[fill] (85,5) circle [radius=0.5];
\draw[fill] (25,5) circle [radius=0.5];

\node[] (C) at (25,-1)
						{$n$};
\draw[](88,58)--(72,42);
\draw[](110,80)--(102,72);
\draw[](58,28)--(40,10);
\draw[](40,10)--(0,50);
\node[] (C) at (85,103)
						{$\frac{n}{2}$};
\draw[](0,50)--(8,58);
\draw[](22,72)--(40,90);
\draw[](40,90)--(58,72);
\draw[](72,58)--(110,20);
\node[] (C) at (-4,50)
						{$2$};
\draw[](110,20)--(100,10);
\draw[](100,10)--(72,38);
\draw[](58,52)--(52,58);
\draw[](38,72)--(30,80);
\draw[](30,80)--(25,95);

\node[] (C) at (85,-3)
						{$\frac{2n-k+3}{2}$};
\draw[](25,95)--(20,80);
\draw[](20,80)--(12,72);
\draw[](12,58)--(50,20);
\draw[](50,20)--(55,5);
\draw[](55,5)--(58,12);
\draw[](72,32)--(98,58);
\draw[](98,72)--(90,80);
\draw[](90,80)--(85,95);
\draw[](85,95)--(80,80);
\draw[](58,58)--(10,10);
\draw[](80,80)--(72,72);
\draw[](10,10)--(0,20);

\node[] (C) at (65,25)
						{$\cdots$};
\node[] (C) at (65,85)
						{$\cdots$};
\node[] (C) at (25,62)
						{$\cdot$};
\node[] (C) at (25,64)
						{$\cdot$};
\node[] (C) at (25,66)
						{$\cdot$};
\node[] (C) at (85,62)
						{$\cdot$};
\node[] (C) at (85,64)
						{$\cdot$};
\node[] (C) at (85,66)
						{$\cdot$};
\node[] (C) at (50,-20)
						{$(2)~k~\mbox{odd},~n~\mbox{even}$};
\end{tikzpicture}}

\end{center}

\begin{center}
{\begin{tikzpicture}[scale=0.05]

\foreach \x in {0,30cm}
\foreach \y in {0,30cm,60cm}
\draw[->,xshift=\x,yshift=\y](4.9,24.9) -- (5,25);

\foreach \x in {0,30cm,60cm}
\foreach \y in {0}
\draw[->,xshift=\x,yshift=\y](19.9,39.9) -- (20,40);

\foreach \x in {0,30cm}
\foreach \y in {0}
\draw[->,xshift=\x,yshift=\y](11.5,31.5) -- (12.5,32.5);

\foreach \x in {0}
\foreach \y in {0,30cm}
\draw[->,xshift=\x,yshift=\y](27.4,47.4) -- (27.5,47.5);
\draw[->](87.4,47.4) -- (87.5,47.5);

\foreach \x in {0}
\foreach \y in {0}
\draw[->,xshift=\x,yshift=\y](27.4,12.2) -- (27.5,12.5);

\foreach \x in {0,30cm,60cm}
\foreach \y in {0,30cm,60cm}
\draw[->,xshift=\x,yshift=\y](14.9,25.1) -- (15,25);

\foreach \x in {0,30cm}
\foreach \y in {0}
\draw[->,xshift=\x,yshift=\y](7.4,32.6) -- (7.5,32.5);

\foreach \x in {0,60cm}
\foreach \y in {0}
\draw[->,xshift=\x,yshift=\y](29.9,40.1) -- (30,40);

\foreach \x in {0,30cm,60cm}
\foreach \y in {0,-30cm}
\draw[->,xshift=\x,yshift=\y](22.4,77.6) -- (22.5,77.5);

\foreach \x in {0,30cm}
\foreach \y in {0,30cm,60cm}
\draw[->,xshift=\x,yshift=\y](5.1,14.9) -- (5,15);

\foreach \x in {0,30cm,60cm}
\foreach \y in {0}
\draw[->,xshift=\x,yshift=\y](20.1,29.9) -- (20,30);

\foreach \x in {0,30cm}
\foreach \y in {0}
\draw[->,xshift=\x,yshift=\y](12.6,37.4) -- (12.5,37.5);

\foreach \x in {0}
\foreach \y in {0,30cm}
\draw[->,xshift=\x,yshift=\y](27.6,22.4) -- (27.5,22.5);
\draw[->](87.6,52.4) -- (87.5,52.5);
\foreach \x in {0,30cm,60cm}
\foreach \y in {0,30cm,60cm}
\draw[<-,xshift=\x,yshift=\y](15,15) -- (15.1,15.1);

\foreach \x in {0,30cm,60cm}
\foreach \y in {0,30cm}
\draw[<-,xshift=\x,yshift=\y](22.5,22.5) -- (22.6,22.6);

\foreach \x in {0,60cm}
\foreach \y in {0}
\draw[->,xshift=\x,yshift=\y](30.1,30.1) -- (30,30);

\foreach \x in {0,30cm}
\foreach \y in {0}
\draw[->,xshift=\x,yshift=\y](7.6,37.6) -- (7.5,37.5);

\foreach \x in {0,30cm}
\draw[<-,xshift=\x] (22.5,12.5)--(22.4,12.8);

\foreach \x in {0}
\draw[->,xshift=\x] (27.5,87.5)--(27.4,87.8);

\draw[->] (89.4,18.3)--(89.2,18.5);
\draw[->] (86.8,15.4)--(87,15.3);

\foreach \x in {0,30cm}
\draw[<-,xshift=\x] (22.1,86.9)--(22.5,87.8);

\draw[->] (89.8,82.2)--(90,82.5);
\draw[->] (87.05,84.6)--(86.9,84.5);

\draw[] (0,5)--(0,58)(0,72)--(0,95);
\draw[] (0,5)--(58,5)(72,5)--(95,5);
\draw[] (0,95)--(58,95)(72,95)--(95,95);
\draw[] (95,5)--(95,58)(95,72)--(95,95);

\draw[] (0,20)--(38,58)(52,72)--(58,78);
\draw[] (72,88)--(88,72);

\node[] (C) at (-4,20)
						{$1$};
\draw[] (88,58)--(72,42)(58,28)--(40,10);
\draw[] (40,10)--(0,50);
\draw[] (0,50)--(8,58)(22,72)--(40,90);
\draw[] (40,90)--(58,72)(72,58)--(95,35);

\node[] (C) at (-4,50)
						{$2$};
\draw[] (95,35)--(72,12);
\draw[] (58,22)--(22,58)(8,72)--(0,80);
\draw[] (0,80)--(10,90);
\draw[](10,90)--(28,72)(42,58)--(58,42)(72,28)--(80,20);
\draw[](80,20)--(95,10);

\node[] (C) at (25,103)
						{$\frac{n-k+3}{2}$};
\draw[](95,10)--(85,25);
\draw[](85,25)--(72,38)(58,52)--(52,58)(38,72)--(30,80);
\draw[](30,80)--(25,95);

\node[] (C) at (90,-2)
						{$\frac{2n-k}{2}$};
\draw[](25,95)--(20,80);
\draw[](20,80)--(12,72);
\draw[](12,58)--(50,20);
\draw[](50,20)--(55,5);
\draw[](55,5)--(58,14);
\draw[](72,32)--(90,50);
\draw[](90,50)--(82,58);
\draw[](58,86)--(55,95);
\draw[](55,95)--(50,80);
\draw[](50,80)--(42,72)(28,58)--(5,35);
\draw[](5,35)--(20,20);
\draw[](20,20)--(25,5);
\draw[](25,5)--(30,20);
\draw[](30,20)--(58,48)(82,72)--(85,75);
\draw[](85,75)--(95,90);

\node[] (C) at (25,-1)
						{$n$};
\draw[](95,90)--(80,80);
\draw[](58,58)--(10,10)(80,80)--(72,72);
\draw[](10,10)--(0,20);

\draw[fill] (95,35) circle [radius=0.5];
\draw[fill] (25,95) circle [radius=0.5];
\draw[fill] (55,95) circle [radius=0.5];
\draw[fill] (95,90) circle [radius=0.5];
\draw[fill] (55,5) circle [radius=0.5];
\draw[fill] (25,5) circle [radius=0.5];
\draw[fill] (95,10) circle [radius=0.5];
\draw[fill] (0,80) circle [radius=0.5];
\draw[fill] (0,50) circle [radius=0.5];
\draw[fill] (0,20) circle [radius=0.5];

\node[] (C) at (90,102)
						{$\frac{n+1}{2}$};

\node[] (C) at (65,25)
						{$\cdots$};
\node[] (C) at (65,85)
						{$\cdots$};
\node[] (C) at (25,62)
						{$\cdot$};
\node[] (C) at (25,64)
						{$\cdot$};
\node[] (C) at (25,66)
						{$\cdot$};
\node[] (C) at (85,62)
						{$\cdot$};
\node[] (C) at (85,64)
						{$\cdot$};
\node[] (C) at (85,66)
						{$\cdot$};

\node[] (C) at (50,-20)
						{$(3)~k~\mbox{even},~n~\mbox{odd}$};
\end{tikzpicture}}
\qquad\qquad
{\begin{tikzpicture}[scale=0.045]

\foreach \x in {0,30cm}
\foreach \y in {0,30cm,60cm}
\draw[->,xshift=\x,yshift=\y](4.9,24.9) -- (5,25);

\foreach \x in {0,30cm,60cm}
\foreach \y in {0,60cm}
\draw[->,xshift=\x,yshift=\y](19.9,39.9) -- (20,40);

\foreach \x in {0,30cm}
\foreach \y in {0,60cm}
\draw[->,xshift=\x,yshift=\y](11.5,31.5) -- (12.5,32.5);

\foreach \x in {0,60cm}
\foreach \y in {0,30cm}
\draw[->,xshift=\x,yshift=\y](27.4,47.4) -- (27.5,47.5);

\foreach \x in {0}
\foreach \y in {0}
\draw[->,xshift=\x,yshift=\y](27.4,12.2) -- (27.5,12.5);

\foreach \x in {0,30cm,60cm}
\foreach \y in {0,30cm,60cm}
\draw[->,xshift=\x,yshift=\y](14.9,25.1) -- (15,25);

\foreach \x in {0,30cm}
\foreach \y in {0}
\draw[->,xshift=\x,yshift=\y](7.4,32.6) -- (7.5,32.5);

\foreach \x in {0,60cm}
\foreach \y in {0,60cm}
\draw[->,xshift=\x,yshift=\y](29.9,40.1) -- (30,40);

\foreach \x in {0,30cm,60cm}
\foreach \y in {0,-30cm}
\draw[->,xshift=\x,yshift=\y](22.4,77.6) -- (22.5,77.5);

\foreach \x in {0}
\foreach \y in {0}
\draw[->,xshift=\x,yshift=\y](37.4,92.6) -- (37.5,92.5);
\foreach \x in {0,30cm}
\foreach \y in {0,30cm,60cm}
\draw[->,xshift=\x,yshift=\y](5.1,14.9) -- (5,15);

\foreach \x in {0,30cm,60cm}
\foreach \y in {0,60cm}
\draw[->,xshift=\x,yshift=\y](20.1,29.9) -- (20,30);

\foreach \x in {0,30cm}
\foreach \y in {0}
\draw[->,xshift=\x,yshift=\y](12.6,37.4) -- (12.5,37.5);

\foreach \x in {0}
\foreach \y in {0,30cm,60cm}
\draw[->,xshift=\x,yshift=\y](27.6,22.4) -- (27.5,22.5);

\foreach \x in {60cm}
\foreach \y in {30cm,60cm}
\draw[->,xshift=\x,yshift=\y](27.6,22.4) -- (27.5,22.5);
\foreach \x in {0,30cm,60cm}
\foreach \y in {0,30cm,60cm}
\draw[<-,xshift=\x,yshift=\y](15,15) -- (15.1,15.1);

\foreach \x in {0,30cm,60cm}
\foreach \y in {0,30cm,60cm}
\draw[<-,xshift=\x,yshift=\y](22.5,22.5) -- (22.6,22.6);

\foreach \x in {0,60cm}
\foreach \y in {0,60cm}
\draw[->,xshift=\x,yshift=\y](30.1,30.1) -- (30,30);

\foreach \x in {0,30cm}
\foreach \y in {0}
\draw[->,xshift=\x,yshift=\y](7.6,37.6) -- (7.5,37.5);
\draw[<-] (37,101)--(37.1,101.3);

\foreach \x in {0,30cm}
\draw[<-,xshift=\x] (22.5,12.5)--(22.4,12.8);

\draw[<-] (42.5,102.5)--(42.6,102.2);

\draw[->] (89.4,18.3)--(89.2,18.5);
\draw[->] (86.8,15.4)--(87,15.3);

\draw[->] (7.7,99.9)--(7.5,100);
\draw[->] (5.25,97.2)--(5.4,97);

\draw[] (0,5)--(0,58)(0,72)--(0,110);
\draw[] (0,5)--(58,5)(72,5)--(95,5);
\draw[] (0,110)--(58,110)(72,110)--(95,110);
\draw[] (95,5)--(95,58)(95,72)--(95,110);

\draw[] (0,20)--(38,58)(52,72)--(58,78)(72,92)--(85,105);
\draw[] (85,105)--(95,95);

\node[] (C) at (-4,20)
						{$1$};
\draw[] (95,95)--(72,72)(58,58)--(10,10);
\draw[] (10,10)--(0,20);

\node[] (C) at (90,117)
						{$\frac{n+2}{2}$};
\draw[] (0,50)--(8,58)(22,72)--(55,105);
\draw[] (55,105)--(58,102)(72,88)--(88,72);

\node[] (C) at (-4,50)
						{$2$};
\draw[] (88,58)--(72,42)(58,28)--(40,10);
\draw[] (40,10)--(0,50);
\draw[] (0,80)--(25,105);
\draw[](25,105)--(58,72)(72,58)--(95,35);
\draw[](95,35)--(72,12);
\draw[](58,22)--(22,58)(8,72)--(0,80);

\node[] (C) at (90,-2)
						{$\frac{2n-k}{2}$};
\draw[](0,105)--(15,95);
\draw[](15,95)--(38,72)(52,58)--(58,52)(72,38)--(85,25);
\draw[](85,25)--(95,10);

\node[] (C) at (5,117)
						{$\frac{n-k+2}{2}$};
\draw[](95,10)--(80,20);
\draw[](80,20)--(72,28)(58,42)--(42,58)(28,72)--(10,90);
\draw[](10,90)--(0,105);
\draw[](40,110)--(45,95);
\draw[](45,95)--(58,82)(82,58)--(90,50);
\draw[](90,50)--(72,32);
\draw[](58,13)--(55,5);
\draw[](55,5)--(50,20);
\draw[](50,20)--(12,58);
\draw[](12,72)--(35,95);
\draw[](35,95)--(40,110);
\draw[](75,95)--(90,80);
\draw[](90,80)--(82,72)(58,48)--(30,20);
\draw[](30,20)--(25,5);
\draw[](72,104)--(75,95);
\draw[](25,5)--(20,20);
\draw[](20,20)--(5,35);
\draw[](5,35)--(28,58)(42,72)--(58,88);

\draw[fill] (0,20) circle [radius=0.5];
\draw[fill] (95,95) circle [radius=0.5];
\draw[fill] (0,50) circle [radius=0.5];
\draw[fill] (95,10) circle [radius=0.5];
\draw[fill] (95,35) circle [radius=0.5];
\draw[fill] (40,110) circle [radius=0.5];
\draw[fill] (0,105) circle [radius=0.5];
\draw[fill] (55,5) circle [radius=0.5];
\draw[fill] (25,5) circle [radius=0.5];

\node[] (C) at (25,-1)
						{$n$};

\node[] (C) at (65,25)
						{$\cdots$};
\node[] (C) at (65,85)
						{$\cdots$};
\node[] (C) at (25,62)
						{$\cdot$};
\node[] (C) at (25,64)
						{$\cdot$};
\node[] (C) at (25,66)
						{$\cdot$};
\node[] (C) at (85,62)
						{$\cdot$};
\node[] (C) at (85,64)
						{$\cdot$};
\node[] (C) at (85,66)
						{$\cdot$};

\node[] (C) at (50,-20)
						{$(4)~k~\mbox{even},~n~\mbox{even}$};
\end{tikzpicture}}
\end{center}
\begin{center}
\caption{Initial Postnikov diagram}\label{Figure:initial Postnikov diagram}
\end{center}
\end{figure}


\begin{figure}
\begin{center}
{\begin{tikzpicture}[scale=0.18]

\foreach \x in {0,10cm,25cm}
\foreach \y in {0,10cm,25cm}
\draw[<-,thick,xshift=\x,yshift=\y](0.5,0) -- (4.5,0);
\foreach \x in {5}
\foreach \y in {0,15cm}
\draw[<-,thick,xshift=\x,yshift=\y](5.5,5) -- (9.5,5);

\foreach \x in {5}
\foreach \y in {0,10cm,25cm}
\draw[->,thick,xshift=\x,yshift=\y](5.5,0) -- (9.5,0);
\foreach \x in {0,10cm,25cm}
\foreach \y in {0,15cm}
\draw[->,thick,xshift=\x,yshift=\y](0.5,5) -- (4.5,5);

\foreach \x in {0,10cm,25cm}
\foreach \y in {5}
\draw[->,thick,xshift=\x,yshift=\y](0,0.5) -- (0,4.5);
\foreach \x in {0,10cm,25cm}
\foreach \y in {0,15cm}
\draw[->,thick,xshift=\x,yshift=\y](5,5.5) -- (5,9.5);

\foreach \x in {0,10cm,25cm}
\foreach \y in {5}
\draw[<-,thick,xshift=\x,yshift=\y](5,0.5) -- (5,4.5);
\foreach \x in {0,10cm,25cm}
\foreach \y in {0,15cm}
\draw[<-,thick,xshift=\x,yshift=\y](0,5.5) -- (0,9.5);
\node[] (C) at (20,22.5)
						{$\cdots$};
\node[] (C) at (20,5)
						{$\cdots$};
\node[] (C) at (7.5,16)
						{$\cdot$};
\node[] (C) at (7.5,15)
						{$\cdot$};
\node[] (C) at (7.5,14)
						{$\cdot$};
\node[] (C) at (27.5,16)
						{$\cdot$};
\node[] (C) at (27.5,15)
						{$\cdot$};
\node[] (C) at (27.5,14)
						{$\cdot$};
\node[] (C) at (35,16)
						{$\cdot$};
\node[] (C) at (35,15)
						{$\cdot$};
\node[] (C) at (35,14)
						{$\cdot$};
\node[] (C) at (36,2.5)
						{R1};
\node[] (C) at (36,7.5)
						{R2};
\node[] (C) at (38,22.5)
						{R(n-k-2)};
\node[] (C) at (0,-1.5)
						{{\qihao $(1,1)$}};
\node[] (C) at (5,-1.5)
						{{\qihao $(2,1)$}};
\node[] (C) at (-3,5)
						{{\qihao $(1,2)$}};
\node[] (C) at (34,27)
						{{\qihao (k-1,n-k-1)}};
\node[] (C) at (0,0)
						{{\qihao$\bullet$}};
\node[] (C) at (0,5)
						{{\qihao$\bullet$}};
\node[] (C) at (0,10)
						{{\qihao$\bullet$}};
\node[] (C) at (0,20)
						{{\qihao$\bullet$}};
\node[] (C) at (0,25)
						{{\qihao$\bullet$}};

\node[] (C) at (5,0)
						{{\qihao$\bullet$}};
\node[] (C) at (5,5)
						{{\qihao$\bullet$}};
\node[] (C) at (5,10)
						{{\qihao$\bullet$}};
\node[] (C) at (5,20)
						{{\qihao$\bullet$}};
\node[] (C) at (5,25)
						{{\qihao$\bullet$}};

\node[] (C) at (10,0)
						{{\qihao$\bullet$}};
\node[] (C) at (10,5)
						{{\qihao$\bullet$}};
\node[] (C) at (10,10)
						{{\qihao$\bullet$}};
\node[] (C) at (10,20)
						{{\qihao$\bullet$}};
\node[] (C) at (10,25)
						{{\qihao$\bullet$}};

\node[] (C) at (15,0)
						{{\qihao$\bullet$}};
\node[] (C) at (15,5)
						{{\qihao$\bullet$}};
\node[] (C) at (15,10)
						{{\qihao$\bullet$}};
\node[] (C) at (15,20)
						{{\qihao$\bullet$}};
\node[] (C) at (15,25)
						{{\qihao$\bullet$}};

\node[] (C) at (25,0)
						{{\qihao$\bullet$}};
\node[] (C) at (25,5)
						{{\qihao$\bullet$}};
\node[] (C) at (25,10)
						{{\qihao$\bullet$}};
\node[] (C) at (25,20)
						{{\qihao$\bullet$}};
\node[] (C) at (25,25)
						{{\qihao$\bullet$}};

\node[] (C) at (30,0)
						{{\qihao$\bullet$}};
\node[] (C) at (30,5)
						{{\qihao$\bullet$}};
\node[] (C) at (30,10)
						{{\qihao$\bullet$}};
\node[] (C) at (30,20)
						{{\qihao$\bullet$}};
\node[] (C) at (30,25)
						{{\qihao$\bullet$}};

\node[] (C) at (2.5,-4)
						{{C1}};
\node[] (C) at (7.5,-4)
						{{C2}};
\node[] (C) at (12.5,-4)
						{{C3}};
\node[] (C) at (27.5,-4)
						{{C(k-2)}};
\end{tikzpicture}}
\end{center}
\begin{center}
\caption{Initial quiver $Q_{\textrm{ini}}$ ($k$ odd, $n$ odd)}\label{Figure:initial quiver1}
\end{center}
\end{figure}
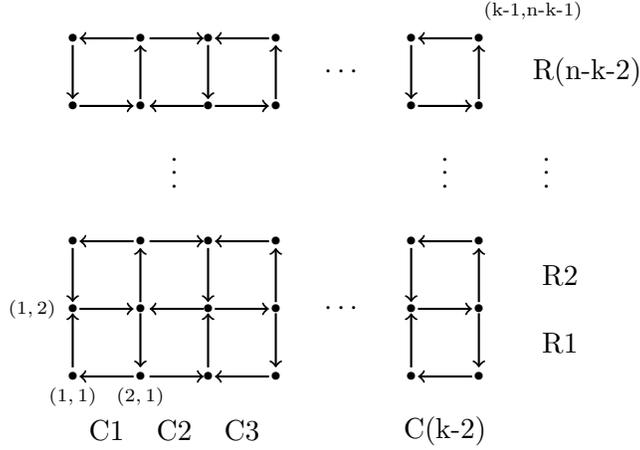

To prove the rigidity and Jacobi-finiteness property of $(\overline{Q},F,\overline{W})$, we need the following lemma.
\begin{lemma}\label{Lemma:first key lemma}
Let $l$ be a cycle in $Q_{\textrm{ini}}$ with end point $(i_0,j_0)$. There exists a positive integer $m$ such that $l-\omega^m\in J(Q_{\textrm{ini}},W_{\textrm{ini}})$ for any fundamental cycle $\omega$ with end point $(i_0,j_0)$.
\end{lemma}
\begin{proof}
Without loss of generality, we may assume that $(i_0,j_0)$ is at the left top corner of $\omega$. So $\omega$ is located at $R(j_0-1)$ and $Ci_0$ of $Q_{\textrm{ini}}$. The proof is proceeded in two steps.

\medskip

{\bf{Step 1:}} We claim that there exists a cycle $\xi$ satisfying the following conditions:
\begin{itemize}
\item[(1)] the end point of $\xi$ is $(i_0,j_0)$;
\item[(2)] $l-\xi\in J(Q_{\textrm{ini}},W_{\textrm{ini}})$;
\item[(3)] any highest vertex of $\xi$ is located at the top of the $j_0$-th level of $Q_{\textrm{ini}}$.
\end{itemize}

Let $a=a(i,j)$ be a highest vertex of $l$. If $j=2$, then $l$ itself already satisfies the conditions of $\xi$. So we assume $j \geq 2$. Then up to the left-right symmetries, we may assume the local configuration of $l$ is as in Figure \ref{Figure:first key lemma}, where the bold arrows form a subpath of a cycle which cyclically equivalent to $l$. Now we construct a new cycle $l'$ from $l$ with end point $(i_0,j_0)$ such that $l-l'\in J(Q_{\textrm{ini}},W_{\textrm{ini}})$.

Note that we may assume that none of the end points of $\gamma$ is $(i_0,j_0)$. Otherwise, $a$ is already located at the top of the $j_0$-th level of $Q_{\textrm{ini}}$, so it is unnecessary to consider such $a$. Therefore $\delta\gamma\beta$ is a subpath of $l$, and we may write $l= q \delta\gamma\beta p$ with $p$ and $q$ the subpaths of $l$, where $p$ and $q$ maybe trivial paths. Let $l''= q\nu\mu\rho p$. Then the end point of $l''$ is still the $(i_0,j_0)$, and $l-l''= q(\delta\gamma\beta-\nu\mu\rho)p=p(\partial_\alpha W_{\textrm{ini}})q\in J(Q_{\textrm{ini}},W_{\textrm{ini}}).$
If $a$ is still a vertex on $l''$, we repeat above construction untill $a$ is never a vertex on a cycle $l'$, which makes sense since the length of $l$ is finite. The final cycle $l'$ is what we want.
Note that the cycle $l'$ has the following properties,
\begin{itemize}
\item[(1)] $l-l'\in J(Q_{\textrm{ini}},W_{\textrm{ini}})$;
\item[(2)] $a$ is never a highest vertex of $l'$;
\item[(3)] no new highest vertex arises in $l'$ with respect to $l$.
\end{itemize}
Thus by  inductively constructing the cycle $l'$, we may find a cycle $\xi$ satisfies the conditions in the claim.

\medskip

{\bf{Step 2:}} For the cycle $\xi$ produced in Step I, we consider the lowest, the leftmost and the rightmost vertices, similar to the analysis used in Step I, we obtain a cycle $\zeta$ such that
\begin{itemize}
\item[(1)] the end point of $\zeta$ is $(i_0,j_0)$;
\item[(2)] $l-\zeta\in J(Q_{\textrm{ini}},W_{\textrm{ini}})$;
\item[(3)] $\zeta$ lies at $R(j_0-1)$ and $Ci_0$, with $width(\zeta)=height(\zeta)=1$.
\end{itemize}

By the item $(3)$, $\zeta$ is a power of a fundamental cycle $\omega'$ lies at $R(j_0-1)$ and $Ci_0$. Since by our assumption of $\omega$, $\omega'$ is a fundamental cycle starting at $(i_0,j_0)$ and lies at $R(j_0-1)$ and $Ci_0$. So we have $\omega'=\omega$ by the item $(1)$. Therefore we have proved that there exists a positive integer $m$ such that $l-\omega^m\in J(Q_{\textrm{ini}},W_{\textrm{ini}})$.

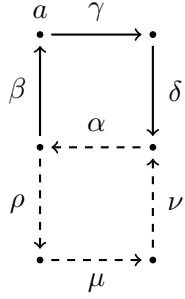
\begin{figure}
\begin{tikzpicture}[scale=0.15]
\draw[fill] (10,5) circle [radius=0.25];
\draw[fill] (20,5) circle [radius=0.25];
\draw[fill] (10,15) circle [radius=0.25];
\draw[fill] (20,15) circle [radius=0.25];
\draw[fill] (10,25) circle [radius=0.25];
\draw[fill] (20,25) circle [radius=0.25];

\draw[->,thick] (10,16) -- (10,24);
\draw[->,thick] (11,25) -- (19,25);
\draw[->,thick] (20,24) -- (20,16);

\draw[->,dashed,thick] (10,14) -- (10,6);
\draw[->,dashed,thick] (11,5) -- (19,5);
\draw[<-,dashed,thick] (20,14) -- (20,6);
\draw[<-,dashed,thick] (11,15) -- (19,15);

\draw (10,27) node {$a$};
\draw (15,27) node {$\gamma$};
\draw (15,17) node {$\alpha$};
\draw (15,3) node {$\mu$};
\draw (8,10) node {$\rho$};
\draw (22,10) node {$\nu$};
\draw (22,20) node {$\delta$};
\draw (8,20) node {$\beta$};
\end{tikzpicture}
\begin{center}
\caption{Local configuration neighbouring the highest vertex $a$ of a cycle}\label{Figure:first key lemma}
\end{center}
\end{figure}

\end{proof}

\begin{Theorem}\label{Theorem£ºrigid}
Any QP $(Q,W)$ of a Grassmannian cluster algebra is rigid.
\end{Theorem}
Because the QP-mutations preserve rigidity,
it suffices to prove the theorem for the initial QP $(Q_{\textrm{ini}},W_{\textrm{ini}})$. So we have to show that any cycle in $Q_{\textrm{ini}}$ is cyclically equivalent to a cycle in the Jacobian idea $J(Q_{\textrm{ini}},W_{\textrm{ini}})$. This is easy for the case $k=2$ or $k=n-2$. Now we assume $k \neq 2$ and $k \neq n-2$. The following lemma is useful.

\begin{lemma}\label{Lemma:second key lemma}
Let $\omega_1$ and $\omega_2$ be two fundamental cycles of $Q_{\textrm{ini}}$ sharing a common arrow $\alpha$. For any positive integer $m$, if $\omega^m_1$ is cyclically equivalent to an element in $J(Q_{\textrm{ini}},W_{\textrm{ini}})$, then $\omega^m_2$ is also cyclically equivalent to an element in $J(Q_{\textrm{ini}},W_{\textrm{ini}})$.
\end{lemma}
\begin{proof}
Recall that $C$ is the closure of the span of all elements of the form
$$\alpha_s\cdots\alpha_2\alpha_1-\alpha_1\alpha_s\cdots\alpha_2,$$ where $\alpha_s\cdots\alpha_2\alpha_1$ is a cycle. Since $\omega^m_1$ is cyclically equivalent to an element in the ideal $J(Q_{\textrm{ini}},W_{\textrm{ini}})$, there is a potential $\omega\in J(Q_{\textrm{ini}},W_{\textrm{ini}})$ such that $\omega^m_2-\omega\in C$. Assume that $\alpha p_1$ (resp. $\alpha p_2$) is the fundamental cycle which is cyclically equivalent to $\omega_1$ (resp. $\omega_2$), where $p_1$ and $p_2$ are paths with head $t(\alpha)$ and tail $h(\alpha)$.
Then we use the partial derivation $\partial_\alpha$ to obtain that
$$\alpha p_1-\alpha p_2\in J(Q_{\textrm{ini}},W_{\textrm{ini}}).$$
Moreover, since $\alpha p_1-\alpha p_2$ is a factor of $(\alpha p_1)^m-(\alpha p_2)^m$,
$$(\alpha p_1)^m-(\alpha p_2)^m\in J(Q_{\textrm{ini}},W_{\textrm{ini}}).$$
Note that $(\alpha p_2)^m-\omega^m_2\in C$ and $\omega^m_2-\omega\in C$, thus $(\alpha p_2)^m-\omega\in C$. Therefore
$$\omega^m_1-[(\alpha p_1)^m-(\alpha p_2)^m+\omega]=[\omega^m_1-(\alpha p_1)^m]+[(\alpha p_2)^m-\omega]\in C,$$
where $(\alpha p_1)^m-(\alpha p_2)^m+\omega\in J(Q_{\textrm{ini}},W_{\textrm{ini}})$. This completes the proof.
\end{proof}
\emph{Proof of the theorem:} we separate the proof into three steps.

\medskip

{\bf{Step 1:}} See the Figure \ref{Figure:initial quiver1}, note that there exists an arrow $\alpha$ and a fundamental cycle $\alpha p$ such that the only fundamental cycles which contain $\alpha$ are those in the cyclically equivalent set $[\alpha p]$. Actually, one may always choose the arrow $\alpha$ from $a(2,1)$ to $a(1,1)$ and the left bottom fundamental cycle of the quiver. So
$$(\alpha p)^m=(\alpha \partial_\alpha {W})^m\in J(Q_{\textrm{ini}},W_{\textrm{ini}})$$
for any positive integer $m$. That means for any fundamental cycle $\omega_1$ in $[\alpha p]$ and any positive integer $m$, $\omega_1^m$ is cyclically equivalent to an element in $J(Q_{\textrm{ini}},W_{\textrm{ini}})$.

\medskip

{\bf{Step 2:}} For any fundamental cycle $\omega_2$ and any positive integer $m$, by recursively using Lemma \ref{Lemma:second key lemma}, we find a fundamental cycle $\omega_1$ appearing in Step I, such that $\omega^m_2$ is cyclically equivalent $\omega^m_1$. Thus $\omega^m_2$ is cyclically equivalent to an element in $J(Q_{\textrm{ini}},W_{\textrm{ini}})$.

\medskip

{\bf{Step 3:}} For any cycle $l$ in $(\overline{Q}_{\textrm{ini}},F)$, by Lemma \ref{Lemma:first key lemma}, there is a power $\omega^m_2$ of fundamental cycle with $l-\omega^{m}_2\in J(Q_{\textrm{ini}},W_{\textrm{ini}})$. By Step II, there is an element $\omega^m_1$ in $J(Q_{\textrm{ini}},W_{\textrm{ini}})$ such that $\omega^m_2-\omega^m_1\in C$. That is
$$l-[l-\omega^m_2+\omega^m_1]=\omega^m_2-\omega^m_1\in C,$$
 where $l-\omega^m_2+\omega^m_1\in J(Q_{\textrm{ini}},W_{\textrm{ini}})$. This means that $l$ is cyclically equivalent to an element in $J(Q_{\textrm{ini}},W_{\textrm{ini}})$. Thus the QP $(Q_{\textrm{ini}},W_{\textrm{ini}})$ is rigid.

\begin{remark}\label{rem:nonrigidity}
As the rigidity for the QP, one may also consider the rigidity for an IQP, see \cite{P18}. Note that $(\overline{Q}_{\textrm{ini}},F,\overline{W}_{\textrm{ini}})$ is not rigid. In particular, any fundamental cycle in $(\overline{Q}_{\textrm{ini}},F,\overline{W}_{\textrm{ini}})$ is not cyclically equivalent to a cycle in $J(\overline{Q}_{\textrm{ini}},F,\overline{W}_{\textrm{ini}})$.
\end{remark}

\begin{Theorem}
For each QP $(Q,W)$ of the Grassmannian cluster algebra, the Jacobian algebra $\mathcal{P}(Q,W)$ is finite dimensional.
\end{Theorem}
\begin{proof}
Since the Jacobi-finiteness of an QP is invariant under QP-mutations, we only prove this for the initial QP $(Q_{\textrm{ini}},W_{\textrm{ini}})$. We have to prove that if the length of a cycle $l$ is large enough, then the cycle belongs to the Jacobian ideal.

By the Lemma \ref{Lemma:first key lemma}, there exist a fundamental cycle $\omega$ and a positive integer $m$ such that $l-\omega^m\in J(Q_{\textrm{ini}},W_{\textrm{ini}})$. On the other hand, note that $length(l)=4 length(\omega)$. So we only need to show that for any fundamental cycle $\omega$, there is a positive integer $n$ such that
$$\omega^n\in J(Q_{\textrm{ini}},W_{\textrm{ini}}).$$
This can be done by iteratively using the relations in $J(Q_{\textrm{ini}},W_{\textrm{ini}})$. For example, we consider $\omega^n$ with $\omega$ shown in Figure \ref{Figure:Jacobi finite}, where the end points of fundamental cycles $\omega$, $\omega_1$ and $\omega_2$ are $a$, $a$, and $b$ respectively.
Then we have
$$\omega^n-\omega_1^n \in J(Q_{\textrm{ini}},W_{\textrm{ini}}),$$
$$\omega_1^n-\delta\omega_2^{n-1}\gamma\beta\alpha \in J(Q_{\textrm{ini}},W_{\textrm{ini}}),$$
and thus
$$\omega^n-\delta\omega_2^{n-1}\gamma\beta\alpha \in J(Q_{\textrm{ini}},W_{\textrm{ini}}).$$
As long as $n$ is large enough, repeating this process, we can find a fundamental cycle $\omega'$ locating at the row $R1$, which belongs to $J(Q_{\textrm{ini}},W_{\textrm{ini}})$, such that
$$\omega^n-q(\omega')^{n'}p \in J(Q_{\textrm{ini}},W_{\textrm{ini}}),$$
where $n'$ is a positive integer, $p$ and $q$ are paths in $Q_{\textrm{ini}}$. Therefore $\omega^n \in J(Q_{\textrm{ini}},W_{\textrm{ini}})$ which completes the proof.

\begin{figure}
\begin{center}
{\begin{tikzpicture}[scale=0.2]
\draw[fill] (0,0) circle [radius=0.2];
\draw[fill] (0,5) circle [radius=0.2];
\draw[fill] (5,5) circle [radius=0.2];
\draw[fill] (5,0) circle [radius=0.2];

\draw[fill] (0,-5) circle [radius=0.2];
\draw[fill] (0,-10) circle [radius=0.2];
\draw[fill] (5,-10) circle [radius=0.2];
\draw[fill] (5,-5) circle [radius=0.2];

\draw[->,thick] (.5,0) -- (4.5,0);
\draw[->,thick] (5,0.5) -- (5,4.5);
\draw[->,thick] (4.5,5) -- (0.5,5);
\draw[->,thick] (0,4.5) -- (0,0.5);

\draw[->,thick] (.5,-10) -- (4.5,-10);
\draw[->,thick] (5,-9.5) -- (5,-5.5);
\draw[->,thick] (4.5,-5) -- (0.5,-5);
\draw[->,thick] (0,-5.5) -- (0,-9.5);

\draw[->,thick] (0,-4.5) -- (0,-0.5);
\draw[<-,thick] (5,-4.5) -- (5,-0.5);

\draw (2.5,-11) node {$\cdot$};
\draw (2.5,-12) node {$\cdot$};
\draw (2.5,-13) node {$\cdot$};

\draw (2.5,2.5) node {$\omega$};
\draw (2.5,-2.5) node {$\omega_1$};
\draw (2.5,-7.5) node {$\omega_2$};

\draw (-2,0) node {$a$};
\draw (-2,-5) node {$b$};
\draw (2.5,0.8) node {$\alpha$};
\draw (-0.8,-2.5) node {$\delta$};
\draw (2.5,-5.8) node {$\gamma$};
\draw (5.8,-2.5) node {$\beta$};

\end{tikzpicture}}
\end{center}
\begin{center}
\caption{Jacobi-finiteness of the QP}\label{Figure:Jacobi finite}
\end{center}
\end{figure}
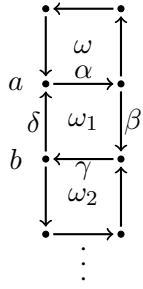
\end{proof}

\begin{remark}\label{rem:nonJacobi-finite}
Unlike the case for the QP $(Q_{\textrm{ini}},W_{\textrm{ini}})$, the IQP $(\overline{Q}_{\textrm{ini}},F,\overline{W}_{\textrm{ini}})$ is not Jacobi-finite. In particular, any power of a fundamental cycle of $\overline{Q}_{ini}$ is non-zero in the Jacobian algebra $\mathcal{P}(\overline{Q}_{\textrm{ini}},F,\overline{W}_{\textrm{ini}})$
\end{remark}




\subsection{The uniqueness}We study in this subsection the uniqueness of the QPs of a Grassmannian cluster algebra. This is based on a general result of Gei\ss-Labardini-Schr\"oer \cite{GLS16}. They give a criterion which guarantees the uniqueness of a non-degenerate QP.

We first recall some definitions in \cite{GLS16}. If $W$ is a {\em finite potential}, i.e. the potential with finite many items in the its expansion, then we denote by $long(W)$ the length of the longest cycle appearing in $W$. For a non-zero element $u\in \C\langle Q\rangle$, denote by $short(u)$ the unique integer such that $u\in \mathfrak{m}^{short(u)}$ but $u\notin \mathfrak{m}^{short(u)+1}$. We also set $short(0)=+\infty$. (see \cite[Section 2.5]{GLS16}). The following two propositions are important for our main result.

\begin{proposition}\cite[Proposition 2.4]{GLS16}\label{Proposition:uniqueness one}
Let $(Q,W)$ be a QP over a quiver $Q$, and let $I$ be a subset of $Q_0$ such that the following hold:
\begin{itemize}
\item[(1)] The full subquiver $Q|_I$
 of $Q$ with vertex set $I$ contains exactly $m$ arrows $\alpha_1,\cdots ,\alpha_m$;
\item[(2)] $l:=\alpha_1\cdots \alpha_m$ is a cycle in $Q$;
\item[(3)] The vertices $s(\alpha_1),\cdots,s(\alpha_m)$ are pairwise different;
\item[(4)] $W$ is non-degenerate.
\end{itemize}
Then the cycle $l$ appears in $W$.
\end{proposition}

\begin{proposition}\cite[Theorem 8.20]{GLS16}\label{Proposition:uniqueness two}
Suppose $(Q,W)$ is a QP over a quiver $Q$ that satisfies the following three properties
\begin{itemize}
\item[(1)] $W$ is a finite potential;
\item[(2)] Every cycle $l$ in $Q$ of length greater then $long(W)$ is cyclically equivalent to an element of the form $\sum_{\alpha\in Q_1}\eta_\alpha\partial_\alpha W$ with $short(\eta_\alpha)+short(\partial_\alpha W)\geqslant length(l)$ for all $\alpha \in Q_1$;
\item[(3)] Every non-degenerate potential on $Q$ is right-equivalent to $W+W'$ for some potential $W'$ with $short(W')> long(W)$.
\end{itemize}
Then $W$ is non-degenerate and every non-degenerate QP on $Q$ is right-equivalent to $W$.
\end{proposition}

\begin{Theorem}\label{thm:uniqueness}
Let $Q$ be a quiver of a Grassmannian cluster algebra, then the QP $(Q,{W})$ is the unique non-degenerate QP on $Q$ up to right-equivalence, and thus the unique rigid QP on $Q$ up to right-equivalence.
\end{Theorem}
\begin{proof}
By Remark \ref{rem:nonrigidity}, $(Q,{W})$ is rigid, so if it is unique as a non-degenerate QP then it must be unique as a rigid QP. Since the mutations of two right-equivalent QPs are still right-equivalent, we only need to prove the theorem for the initial QP.

To do this, we check that the conditions $(1)$-$(3)$ in Proposition \ref{Proposition:uniqueness two} hold for $(Q_{\textrm{ini}},{W}_{\textrm{ini}})$. The condition $(1)$ is clear. Since the cluster algebra of $Gr(2,n)$ is of acyclic type, so there is a unique rigid QP. Otherwise, there exists at least one internal fundamental cycle on $Q_{\textrm{ini}}$, and $long({W}_{\textrm{ini}})=4$. We prove the condition $(2)$ in two steps, see the Figure \ref{Figure:initial quiver1}.

\medskip

{\bf{Step I:}} Let $\omega$ be a fundamental cycle of $Q_{\textrm{ini}}$ and $m$ be a positive integer number. We claim that $\omega^m$ is cyclically equivalent to $\sum_{\alpha\in Q_1} \eta_\alpha\partial_\alpha{W}_{\textrm{ini}}$, where the length of a path appearing in non-zero $\eta_\alpha$ is $4m-3$, and the length of all paths appearing in $\partial_\alpha {W}_{\textrm{ini}}$ is $3$.

We prove this by induction on the level of $\omega$. Assume the level of $\omega$ is $2$ and $\alpha$ be the bottom arrow of $\omega$, then it is cyclically equivalent to $\alpha\partial_\alpha{W}_{\textrm{ini}}$. Moreover,
$$\omega^m \mbox{~ is~ cyclically~ equivalent~ to ~}((\alpha\partial_\alpha {W}_{\textrm{ini}})^{m-1}\alpha)\partial_\alpha {W}_{\textrm{ini}},$$
where $\eta_\alpha=(\alpha\partial_\alpha {W}_{\textrm{ini}})^{m-1}\alpha$ and $\partial_\alpha {W}_{\textrm{ini}}$ satisfy the conditions in the claim.

Now let $\omega$ be a fundamental cycle located at level $t$. Assume that the claim holds for the fundamental cycle $\alpha\rho\nu\mu$, which is located at level $t-1$, see the Figure \ref{Figure:uniqueness one}. Here we only consider the clockwise cycle $\alpha\rho\nu\mu$, another case is similar. So we may assume that $(\alpha\rho\nu\mu)^m$ is cyclically equivalent to a potential $\sum_{\alpha'\in Q_1} \eta_{\alpha'}\partial_{\alpha'}{W}_{\textrm{ini}}$ satisfying the claim.
Then $\omega^m$ is cyclically equivalent to $(\alpha\delta\gamma\beta)^m$, which equals to $(\alpha\rho\nu\mu-\alpha\partial_\alpha {W}_{\textrm{ini}})^m$.

Note that we may write the expansion of
$(\alpha\rho\nu\mu-\alpha\partial_\alpha {W}_{\textrm{ini}})^m$ as the form of
$(\alpha\rho\nu\mu)^m+\sum_k S_k$,
where $S_k$ is a multiplication of $\alpha\rho\nu\mu$ and $-\alpha\partial_\alpha {W}_{\textrm{ini}}$ with the term $-\alpha\partial_\alpha {W}_{\textrm{ini}}$ appearing in it at least once. We write $S_k=S'_k\alpha\partial_\alpha {W}_{\textrm{ini}}S''_k$, where $S'_k$ and $S''_k$ are multiplications (maybe empty) of $\alpha\rho\nu\mu$ and $-\alpha\partial_\alpha {W}_{\textrm{ini}}$. Then $S_k-S''_kS'_k\alpha\partial_\alpha {W}_{\textrm{ini}} \in C$.
Thus
\begin{equation*}
\begin{aligned}
 & (\alpha\delta\gamma\beta)^m-[\sum_{\alpha'\in Q_1} \eta_{\alpha'}\partial_{\alpha'}{W}_{\textrm{ini}}+({\textstyle\sum}_k S''_kS'_k\alpha)\partial_\alpha {W}_{\textrm{ini}}]\\
=&[(\alpha\rho\nu\mu)^m+{\textstyle\sum}_k S_k]-[\sum_{\alpha'\in Q_1} \eta_{\alpha'}\partial_{\alpha'}{W}_{\textrm{ini}}+({\textstyle\sum}_k S''_kS'_k\alpha)\partial_\alpha {W}_{\textrm{ini}}] \\
=& [(\alpha\rho\nu\mu)^m-\sum_{\alpha'\in Q_1} \eta_{\alpha'}\partial_{\alpha'}{W}_{\textrm{ini}}]+{\textstyle\sum}_k (S_k-S''_kS'_k\alpha\partial_\alpha {W}_{\textrm{ini}})
\in C.\\
\end{aligned}
\end{equation*}
So $(\alpha\delta\gamma\beta)^m$, and therefore $\omega^m$, is cyclically equivalent to
$$\sum_{\alpha'\in Q_1} \eta_{\alpha'}\partial_{\alpha'}{W}_{\textrm{ini}}+({\textstyle\sum}_k S''_kS'_k\beta)\partial_\beta {W}_{\textrm{ini}},$$
which satisfies the conditions in the claim.

To sum up, for any fundamental cycle $\omega$ and any positive integer number $m$, $\omega^m$ is cyclically equivalent to $\sum \eta_\alpha\partial_\alpha{W}_{\textrm{ini}}$, where $\eta_\alpha=0$ or each path in $\eta_\alpha$ has length $4m-3$, and each path in $\partial_\alpha {W}_{\textrm{ini}}$ has length $3$. So $short(\eta_\alpha)=+\infty$ or $4m-3$, and $short(\partial_\alpha {W}_{\textrm{ini}})=3$.
Therefore
$$short(\eta_\alpha)+short(\partial_\alpha {W}_{\textrm{ini}})\geqslant 4m = length(\omega^m),$$
and the condition $(2)$ holds for $\omega^m$.

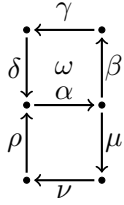
\begin{figure}
\begin{center}
{\begin{tikzpicture}[scale=0.2]
\draw[fill] (0,0) circle [radius=0.2];
\draw[fill] (0,5) circle [radius=0.2];
\draw[fill] (5,5) circle [radius=0.2];
\draw[fill] (5,0) circle [radius=0.2];

\draw[fill] (0,-5) circle [radius=0.2];
\draw[fill] (5,-5) circle [radius=0.2];

\draw[->,thick] (.5,0) -- (4.5,0);
\draw[->,thick] (5,0.5) -- (5,4.5);
\draw[->,thick] (4.5,5) -- (0.5,5);
\draw[->,thick] (0,4.5) -- (0,0.5);

\draw[->,thick] (4.5,-5) -- (0.5,-5);
\draw[->,thick] (0,-4.5) -- (0,-0.5);
\draw[<-,thick] (5,-4.5) -- (5,-0.5);

\draw (2.5,2.5) node {$\omega$};

\draw (2.5,0.8) node {$\alpha$};
\draw (-0.8,2.5) node {$\delta$};
\draw (2.5,6.2) node {$\gamma$};
\draw (5.8,2.5) node {$\beta$};

\draw (2.5,-5.8) node {$\nu$};
\draw (5.8,-2.5) node {$\mu$};
\draw (-0.8,-2.5) node {$\rho$};
\end{tikzpicture}}
\end{center}
\begin{center}
\caption{Uniqueness of the QP}\label{Figure:uniqueness one}
\end{center}
\end{figure}

\medskip

{\bf{Step II:}} Let $l$ be a cycle of $Q$. We use the notations appearing in Lemma \ref{Lemma:first key lemma}. In particular, $l'$ is the new cycle which shares an arrow $\alpha$ with $l$, $p$ and $q$ are two subpaths of $l$ such that $l=l'\pm q \partial_\alpha {W}_{\textrm{ini}}p$. At last we find a fundamental cycle $\omega$ with
$$l-\omega^m \in J(Q,{W}_{\textrm{ini}}).$$
 Assume that $l'$ is cyclically equivalent to $\sum \eta_{\alpha'}\partial_{\alpha'}{W}_{\textrm{ini}}$ and condition $(2)$ holds for $l'$, that is,
$$short(\eta_{\alpha'})+short(\partial_{\alpha'} W)\geqslant length(l').$$
Note that the quiver we consider is the principal part $Q$, so (6) and (7) in Figure \ref{Figure:first key lemma} are the only cases we should deal with. Thus we have
$$length(l)=length(l')\mbox{~and~}length(l)=length(pq)+short(\partial_\alpha {W}_{\textrm{ini}}).$$
Therefore $l$ is cyclically equivalent to $\sum \eta_{\alpha'}\partial_{\alpha'}{W}_{\textrm{ini}}\pm pq \partial_\alpha {W}_{\textrm{ini}}$, which satisfies the condition $(2)$. This proves the condition $(2)$ for all cycles over $Q$.

Finally, the condition $(3)$ follows immediately from the following two observations. By the Proposition \ref{Proposition:uniqueness one}, all of the fundamental cycles appear in ${W}_{\textrm{ini}}$. For any cycle $l$, excepting the fundamental cycles, $length(l)> 4=long({W}_{\textrm{ini}})$.
\end{proof}



\section{Applications}\label{sec:application}

\subsection{Categorification}
An ``additive categorification" of a cluster algebra has been well studied in recent years. Roughly speaking, it lifts a cluster algebra structure on a category level, that is, one may find a {\it cluster structure} (see \cite{BIRS09} for precise definition) on the category. Such category always has a duality property called $2$-Calabi-Yau property. In particular, the cluster category is an important example of $2$-Calabi-Yau triangulated category with cluster structure, which gives a categorification for the cluster algebra of acyclic type with trivial coefficients. In \cite{A09}, for a quiver with potential $(Q,W)$, Amiot constructed a generalized cluster category $\mathcal{C}_{(Q,W)}$.

Some stably $2$-Calabi-Yau Frobenius category also has cluster structure (see \cite{BIRS09,FK09}), which gives categorification of a cluster algebra with non-trivial coefficients.
In our context, such Frobenius category is always a kind of subcategory of module categories. For the cluster algebra structure on the coordinate ring
\begin{equation}\label{eq:CN}
\C[\Gr(k,n)] / (\phi_{\{1,2,\ldots ,k\}} - 1)
\end{equation}
of the affine open cell in the Grassmannian, where $\phi_{\{1,2,\ldots ,k\}}$ is the consecutive Pl\"ucker coordinate indexed by $k$-subset $\{1,2,\ldots ,k\}$, Geiss-Leclerc-Schr\"oer have given in \cite{GLS08} a categorification in terms of a subcategory
$\Sub Q_k$ of the category of finite dimensional modules over the preprojective algebra of type $A_{n-1}$.
Note that the cluster coefficient $\phi_{\{1,2,\ldots ,k\}}$ in $\C[\Gr(k,n)]$ is not realised in the category.
More recently, Jensen-King-Su~\cite{JKS16} have given a full and direct categorification of the cluster structure on $\C[\Gr(k,n)]$,
using the category $\CM(B)$ of (maximal) Cohen-Macaulay modules over the completion of an algebra $B$,
which is a quotient of the preprojective algebra of type $\tilde{A}_{n-1}$.

\begin{remark}
It has been proved in \cite{BKM16} that for a cluster tilting object $T$ in $CM(B)$ corresponding to a Postnikov diagram, the cluster-tilted algebra $End(T)$ is isomorphic to the Jacobian algebra $\mathcal{P}(\overline{Q},F,\overline{W})$. Note that $CM(B)$ is Hom-infinite and $End(T)$ is of infinite dimension, this is compatible with the Hom-infiniteness of $\mathcal{P}(\overline{Q},F,\overline{W})$, see Remark \ref{rem:nonJacobi-finite}.

On the other hand, Amiot, Reiten, and Todorov showed in \cite{ART11} that the generalized cluster category has some ``ubiquity" (see also in \cite{A11,AIR15,Y18}). In our situation, this means the stable categories of $Sub Q_k$ and $\CM(B)$ are both equivalent to a generalized cluster category, which is exactly the generalized cluster category defined by the QP $(Q,W)$.

\end{remark}

\subsection{Auto-equivalence groups and cluster automorphism groups}\label{Section:auto-equivalence groups and cluster automorphism groups}

Recall that for a cluster algebra $\A$, we call an algebra automorphism $f$ a \emph{cluster automorphism}, if it maps a cluster $\x$ to a cluster $\x'$, and is compatible with the mutations of the clusters. Equivalently, an algebra automorphism $f$ is a cluster automorphism if and only if $Q'\cong Q$ or $Q'\cong Q^{op}$, where $Q'$ and $Q$ are the associated quivers of $\x'$ and $\x$ respectively. We refer to \cite{ASS12,CZ16,CZ16b} for the details of cluster automorphisms.

Let $\mathcal{C}$ be a $2$-Calabi-Yau triangulated category with cluster structure. In particular, there is a cluster tilting object $T$ and a \emph{cluster map} $\phi$ which sends cluster tilting objects, which are reachable by iterated mutations from $T$ in category $\mathcal{C},$ to clusters in algebra $\A_{\phi(T)}$, where $\A_{\phi(T)}$ is the cluster algebra with initial cluster $\phi(T)$. In fact $\A_{\phi(T)}$ is the cluster algebra defined by the Gabriel quiver of $End_{\mathcal{C}}(T)$.

Denote by $Aut_T(\mathcal{C})$ a quotient group consisting of the (covariant and contravariant) triangulated auto-equivalence on $\mathcal{C}$ that maps $T$ to a cluster tilting object which is reachable from $T$ itself, where we view two equivalences $F$ and $F'$ the same if $F(T)\cong F(T')$.

Let $F$ be an auto-equivalence in $Aut_T(\mathcal{C})$. Denote by $Q$ and $Q'$ the Gabriel quivers of $End_{\mathcal{C}}(T)$ and $End_{\mathcal{C}}(F(T))$ respectively. Then $Q$ is naturally isomorphic to $Q'$ since $F$ is a triangulated equivalence. Moreover, since $F(T)$ is reachable from $T$, $\phi(F(T))$ is a cluster in $\A_{\phi(T)}$, so there is a cluster automorphism $f$ in $Aut(\A_{\phi(T)})$ which maps $\phi(T)$ to $\phi(F(T))$. Thus $Aut_T(\mathcal{C})$ can be viewed as a subgroup of $Aut(\A_{\phi(T)})$. Conversely, we have the following
\begin{conjecture}\label{conj:auto-equi-gp=clu-auto-gp}
There is an natural isomorphism $Aut_T(\mathcal{C})\cong Aut(\A_{\phi(T)})$.
\end{conjecture}

If $\mathcal{C}$ is algebraic and $Q$ is acyclic, then Keller-Reiten proved in \cite{KR08} that $\mathcal{C}$ is a (classical) cluster category. Then the conjecture has been verified in \cite[section 3]{ASS12} and \cite[Theorem 2.3]{BIRS11}.
For the case of generalized cluster categories, the conjecture is related to the following conjecture, which says that the quivers determine the potentials up to right equivalences.
\begin{conjecture}\label{conj:quiver-determines-potential}
Let $(Q,W)$ be a non-degenerate QP. Assume that $(Q',W')$ is a QP which is mutation equivalent to $(Q,W)$. Then

(1) $(Q',W')$ is right equivalent to $(Q,W)$ if $Q'\cong Q$;

(2) $(Q',W')$ is right equivalent to $(Q^{op},W^{op})$ if $Q'\cong Q^{op}$.
\end{conjecture}

\begin{proposition}\label{prop:auto-equi-gp=clu-auto-gp}
If Conjecture \ref{conj:quiver-determines-potential} is true for a Jacobian-finite QP $(Q,W)$, then Conjecture \ref{conj:auto-equi-gp=clu-auto-gp} is true for the generalized cluster category $\mathcal{C}_{(Q,W)}$.
\end{proposition}
\begin{proof}
Since $(Q,W)$ is Jacobian-finite, recall from \cite{A09} that, there is a canonical cluster titling object $T$ in $\mathcal{C}_{(Q,W)}$ whose endomorphism algebra is isomorphic to the Jacobi-algebra $J(Q,W)$.
 Because we already have
 $$Aut_T(\mathcal{C}_{(Q,W)})\subset Aut(\A_{\phi(T)}),$$
it suffices to show that any cluster automorphism $f$ can be lifted as an auto-equivalence on $\mathcal{C}$ which maps the canonical cluster tilting object to a reachable one. Assume that $f$ maps $\phi(T)$ to a cluster $\mu(\phi(T))$ with quiver $Q'\cong Q$, where $\mu(\phi(T))$ is obtained from $\phi(T)$ by iterated mutations. Denote by $(Q',W')=\mu(Q,W)$ the QP obtained from $(Q,W)$ by the same steps of mutations.

On the one hand, by \cite[Theorem 3.2]{KY11}, there is an equivalence $\Phi$ from $\mathcal{C}_{(Q,W)}$ to $\mathcal{C}_{(Q',W')}$ which maps $T$ to $\mu(T')$, where $T'$ is the canonical cluster tilting object in $\mathcal{C}_{(Q',W')}$ whose endomorphism algebra is isomorphic to $J(Q',W')$.

On the other hand, Conjecture \ref{conj:quiver-determines-potential} ensures that there is a right equivalence between $(Q',W')$ and $(Q,W)$, and then by \cite[Lemma 2.9]{KY11}, there is a covariant equivalence $\Psi$ from $\mathcal{C}_{(Q',W')}$ to $\mathcal{C}_{(Q,W)}$. Note that $\Psi$ maps $T'$ to $T$, and thus maps $\mu(T')$ to $\mu(T)$, since the mutations are obtained by exchanged triangles (see \cite{BIRS09} for example) and the equivalence $\Psi$ is triangulated. Finally, the auto-equivalence $\Psi\Phi$ is what we wanted, which gives a lift of $f$.
We have a similar proof for the case $Q'\cong Q^{op}$. See the following diagram for the equivalences.
$$\xymatrix{\mathcal{C}_{(Q,W)}\ar@{-->}[dr]_{\Psi\Phi}\ar[rr]^\Phi&&\mathcal{C}_{(Q',W')}\ar[dl]^\Psi
\\&\mathcal{C}_{(Q,W)}&}$$
\end{proof}

\begin{Theorem}\label{thm:auto-equi-gp=clu-auto-gp}
For the non-degenerate QPs arising from the Grassmannians cluster algebra,  Conjecture \ref{conj:quiver-determines-potential} is true. So for the associated generalized cluster category $\mathcal{C}$, we have an isomorphism $Aut_T(\mathcal{C})\cong Aut(\A_{\phi(T)})$.
\end{Theorem}
\begin{proof}
Let $(Q,W)$ be a non-degenerate QPs of the Grassmannians cluster algebra, and let $(Q',W')$ be a QP which is mutation equivalent to $(Q,W)$. Then $(Q',W')$ is non-degenerate. On the other hand, by Theorem \ref{thm:uniqueness}, $(Q,W)$ is the unique non-degenerate QP on $Q$, up to right equivalence. So $(Q',W')$ is right equivalent to $(Q,W)$ if $Q\cong Q'$. Note that $(Q^{op},W^{op})$ also has the non-degenerate uniqueness property since $(Q,W)$ has. Thus similarly $(Q',W')$ is also right equivalent to $(Q^{op},W^{op})$, if $Q'\cong Q^{op}$. So the conjecture \ref{conj:quiver-determines-potential} is true, and $Aut_T(\mathcal{C})\cong Aut(\A_{\phi(T)})$.
\end{proof}

\begin{remark}
For the QPs arising from a marked Riemann surface with some "technical conditions", \cite[Theorem 1.4]{GLS16} ensures the non-degenerate uniqueness. So we have a similar isomorphism as in Theorem \ref{thm:auto-equi-gp=clu-auto-gp} for this case.
\end{remark}
\subsection*{Acknowledgments}
The authors would like to thank Dong Yang and Bin Zhu for useful discussions, Claire Amiot and Jeanne Scott for explaining their works in \cite{AIR15} and \cite{S06} respectively, Bernhard Keller and Osamu Iyama for pointing us to the references \cite{A09,AIR15}. Part of the work was done when the first author visited University of Connecticut from 2017 to 2018, he thanks the university and especially professor Ralf Schiffler for hospitality and providing him an excellent working environment.


\end{document}